\documentclass[12pt,a4]{elsarticle}
\usepackage{lipsum}
\usepackage{amsmath}
\usepackage{epsfig,amsthm}
\usepackage{amssymb,amsfonts}
\usepackage{dsfont}
\usepackage{graphicx}
\usepackage{subfigure}
\usepackage{indentfirst}
\usepackage{mathrsfs}
\usepackage{url}
\usepackage{longtable}
\usepackage{eufrak}
\usepackage{todonotes}
\usepackage{mathptmx}      
\makeatletter

\def\ps@pprintTitle{%
 \let\@oddhead\@empty
 \let\@evenhead\@empty
 \def\@oddfoot{}%
 \let\@evenfoot\@oddfoot}
\makeatother

\DeclareMathOperator*{\esssup}{ess\,sup}
\newcommand{\D}{\mathrm{d}}

\newcommand{\e}{\varepsilon}

\newtheorem{thm}{Theorem}

\newtheorem{cor}{Corollary}

\newtheorem{lem}{Lemma}
\newtheorem{prop}{Proposition}
\newtheorem{rmk}{Remark}

\newtheorem{exam}{Example}

\begin{document}
\allowdisplaybreaks
\begin{frontmatter}

\title{From atomistic model to the Peierls--Nabarro model with $\gamma$-surface for dislocations}
\author[a]{Tao Luo}\ead{tluoaa@ust.hk}
\author[b]{Pingbing Ming}\ead{mpb@lsec.cc.ac.cn}
\author[a]{Yang Xiang}\ead{maxiang@ust.hk}
\address[a]{Department of Mathematics, The Hong Kong University of Science and Technology, Clear Water Bay, Kowloon, Hong Kong}
\address[b]{LSEC, Institute of Computational Mathematics and Scientific/Engineering Computing, AMSS,
Chinese Academy of Sciences, No. 55, East Road Zhong-Guan-Cun, Beijing, 100190, China}

\begin{abstract}
The Peierls--Nabarro (PN) model for dislocations is a hybrid model that incorporates the atomistic information of the dislocation core structure into the continuum theory. In this paper, we study the convergence from a full atomistic model to the PN model with $\gamma$-surface for the dislocation in a bilayer system (e.g. bilayer graphene). We prove that the displacement field of and the total energy of the dislocation solution of the PN model are asymptotically close to those of the full atomistic model. Our work can be considered as a generalization of the analysis of the convergence  from atomistic model to Cauchy--Born rule for crystals without defects in the literature.
\end{abstract}


\end{frontmatter}

\section{Introduction}
Dislocations are line defects and the primary carriers of plastic deformation in crystals. They are essential in the understanding of mechanical and plastic properties of crystalline materials \cite{Hirth1982-p-}. Models at different length and time scales have been developed to characterize the behaviors of dislocations and properties of the materials. Atomistic models and first principles calculations are able to capture detailed information of dislocations; however, they are computationally time-consuming and are limited to domains of small size over short time scales. On the other hand, continuum theory of dislocations based on linear elasticity theory  applies to much larger domains; although this theory is accurate outside the dislocation core region (of a few lattice constants size), it breaks down inside the dislocation core where the atomic structure is heavily distorted.
The Peierls--Nabarro (PN) model \cite{Peierls1940-p34-37,Nabarro1947-p256-272} is a hybrid model that incorporates in the continuum model the dislocation core structure informed by atomistic or first principles calculations. Ever since its development, this model and its generalizations have been widely employed in the investigation of dislocation-core related properties \cite{Eshelby1949-p903-912,Vitek1968-p773-786,Vitek1971-p493-508,Kaxiras1993-p3752-3755,Schoeck1994-p1085-1095,Bulatov1997-p4221-4223,Movchan1998-p373-396,Miller1998-p1845-1867,Hartford1998-p2487-2496,Schoeck1999-p2310-2313,Schoeck1999-p2629-2636,Lu2000-p3099-3108,Xu2000-p605-611,Koslowski2002-p2597-2635,Movchan2003-p569-587,Lu2003-p3539-3548,Shen2004-p683-691,Xiang2006-p383-424,Xiang2008-p1447-1460,Wei2008-p275-293,Wei2009-p2333-2354,Dai2011-p438-441,Dai2013-p1327-1337,Dai2014-p162-174,Shen2014-p125-131,Zhou2015-p155438-155438,Mianroodi2015-p109-122,Wang2015-p3768-3784,Wu2016-p11137-11142,Wang2015-p7853-7853,Dai2016-p85410-85410,Dai2016-p5923-5927}.

In the classical PN model \cite{Peierls1940-p34-37,Nabarro1947-p256-272}, the slip plane of a straight edge or screw dislocation divides the crystal into two half-space elastic continua reconnected by a nonlinear potential force incorporating the atomistic effect. The nonlinear potential force is described based on the relative displacement (disregistry) across the slip plane, in the direction of Burgers vector of the dislocation.
The total energy consists of two half-space elastic energies and a misfit energy that leads to the nonlinear potential force across the slip plane. The misfit energy in the classical PN model is approximated by a sinusoidal function of the disregistry.
The dislocation configuration is regarded as the minimizer of the total energy subject to the constraint of the Burgers vector of the dislocation. Such a hybrid model is able to give fairly good results of the dislocation core structure, the non-singular stress field and the total energy, as well as the Peierls stress and the Peierls energy for the motion of the dislocation.

Vitek \cite{Vitek1968-p773-786} introduced the concept of the generalized stacking fault energy (or the $\gamma$-surface), which is expressed in terms of the disregistry vector (relative displacement vector) across the slip plane. For a given disregistry vector, the value of the $\gamma$-surface  is defined as the energy increment per unit area (after relaxation) when the two half-spaces of the crystal have a uniform relative shift across the slip plane by this disregistry vector, which can be calculated by atomistic models. The $\gamma$-surface does not only provide a more realistic nonlinear potential than the sinusoidal form used in the original works of Peierls and Nabarro \cite{Peierls1940-p34-37,Nabarro1947-p256-272}, but also enables vector-valued disregistry function across the slip plane than the scalar disregistry function in the original PN model. Thus it is able to describe the partial dissociation of perfect dislocations \cite{Vitek1968-p773-786,Vitek1971-p493-508}. The $\gamma$-surfaces can be calculated using the empirical potentials as in the original work of Vitek \cite{Vitek1968-p773-786}. Recently, the $\gamma$-surfaces are also obtained more accurately by using the first principles calculations (e.g. \cite{Kaxiras1993-p3752-3755,Bulatov1997-p4221-4223,Hartford1998-p2487-2496,Lu2000-p3099-3108,Zhou2015-p155438-155438}).
The method of $\gamma$-surface has become an important tool for the study of dislocations and plastic properties in crystals.

Besides the incorporation of $\gamma$-surfaces, a considerable amount of generalizations of the classical PN model in other aspects have also been developed in the past seventy years. These generalizations further considered elastic anisotropy \cite{Eshelby1949-p903-912,Schoeck1994-p1085-1095,Xiang2008-p1447-1460}, the lattice discreteness and Peierls stress \cite{Bulatov1997-p4221-4223,Movchan1998-p373-396,Schoeck1999-p2629-2636,Lu2000-p3099-3108,Wei2008-p275-293,Wei2009-p2333-2354,Shen2014-p125-131}, nonlocal misfit energy \cite{Miller1998-p1845-1867,Schoeck1999-p2310-2313} and gradient energy \cite{Wang2015-p3768-3784,Mianroodi2015-p109-122}, and dislocation cross-slip \cite{Lu2003-p3539-3548,Wu2016-p11137-11142}. Generalized PN models have also been developed for curved dislocations \cite{Xu2000-p605-611,Koslowski2002-p2597-2635,Movchan2003-p569-587,Xiang2008-p1447-1460} and within the phase field framework for curved dislocations \cite{Shen2004-p683-691}.
Models within the PN framework have also been proposed for
grain boundaries \cite{Dai2013-p1327-1337,Dai2014-p162-174,Shen2014-p125-131,Wang2015-p7853-7853}, twin boundary junctions \cite{Dai2011-p438-441}, and bilayer graphene and other bilayer materials \cite{Zhou2015-p155438-155438,Dai2016-p85410-85410,Dai2016-p5923-5927}.
The PN models also provide a basis for asymptotic analysis \cite{Xiang2009-p728-743} and rigorous analysis \cite{Garroni2005-p1943-1964,Conti2011-p779-819,Conti2016-p240-251} for obtaining models of dislocation distributions at larger length scales.

Despite the wide range of generalizations and applications of the PN models,
there is not much mathematical understanding and rigorous analysis of these models. Especially, there is no rigorous analysis available in the literature for the fundamental question of convergence from atomistic model to the PN model, to the best of our knowledge. An attempt was made by Fino et al. \cite{Fino2012-p258-293}
to prove the convergence from the nearest neightbor Frenkel--Kontorova model \cite{Frenkel1938-p1340-1348} to the PN model using viscosity solutions. Although discrete lattice-site interactions in the upper and lower half-spaces were included in the nearest Frenkel--Kontorova model adopted in \cite{Fino2012-p258-293}, they directly used a continuum $\gamma$-surface in their Frenkel--Kontorova model without convergence proof from atomistic model. Rigorous convergence analysis from  fully atomistic model to the PN model with justification of the $\gamma$-surface is still lacking.

In this paper, we perform a rigorous analysis for the convergence from atomistic model to the PN model with $\gamma$-surface, in the regime where the lattice constant (or equivalently, the length of the Burgers vector of the dislocation) is much smaller than the length scale of the PN model. As a result, the decomposition of the total energy into the elastic energy and misfit energy (expressed in terms of the $\gamma$-surface) in the framework of the PN models is rigorously justified based on the atomistic model, which has never been done in the literature.
In our proof, we focus on the one-dimensional form of the generalized PN model recently developed for the inter-layer dislocations in bilayer graphene \cite{Dai2016-p85410-85410,Dai2016-p5923-5927}. Note that in the generalized PN model in Refs.~\cite{Dai2016-p85410-85410,Dai2016-p5923-5927}, dislocations are lines lying between the two graphene layers, which are different from the dislocations as point defects in a monolayer graphene studied by Ariza et al. \cite{Ariza2010-p710-734,Ariza2012-p2004-2021} using a discrete dislocation dynamics model.

Our work can also be considered as an extension of the analysis of the convergence issue of Cauchy--Born rule
\cite{Born1954-p-,Blanc2002-p341-381} for elastic media without dislocations and other defects, see, e.g. \cite{Braides1999-p23-58,Blanc2002-p341-381,Friesecke2002-p445-478,Conti2005-p515-530,W.2007-p241-297,W.2010-p1432-1468,Lu2013-p83-108,Ortner2013-p1025-1073,Makridakis2013-p813-843} for the recent progress. The major
difficulty in the analysis of the PN model lies in the fact that due to the presence of
the dislocation, the displacement vector across the slip plane of the dislocation is no
longer continuous, which is unlike in the Cauchy--Born rule where the displacement and its gradient are always continuous. Such a discontinuity in the PN model is
handled by the $\gamma$-surface, and our work successfully establishes the convergence from
atomistic model to the PN model under the one-dimensional setting.
Our proof is inspired by the work of E and Ming \cite{W.2007-p241-297}, in which the stability and convergence of the Cauchy--Born rule were rigorously analyzed for states close to perfect lattices. More precisely, we show that the dislocation solution and the associated energy of the PN model are approximations to those using the full atomistic model.  An important assumption in our analysis is that the ratio of the lattice constant to the dislocation core size is small, which is valid in the bilayer graphene due to the strong intra-layer atomic interaction and weak inter-layer atomic interaction \cite{Dai2016-p85410-85410,Dai2016-p5923-5927}.

Our convergence result is based on the consistency, the linear stability, and a fixed point argument. Infinite interaction range causes difficulties in estimating the truncation error and proving the compactness for the fixed point iteration. This is solved by detailed estimates on the decaying of the derivatives of the pair potentials and the PN solution. Another difficulty is that the stability of the atomistic dislocation solution cannot be directly obtained from that of a perfect lattice because the disregistry might be as large as a (half) Burger vector. This is different from the situation in the Cauchy--Born rule \cite{W.2007-p241-297}, where both atomistic and continuous configurations are perturbed from a common equilibrium state.
To overcome this, we first prove the stability for the PN solution using the standard techniques in elliptic partial differential equations. Consequently, we obtain the first positive eigenvalue of the linearized PN operator at the PN solution. The stability of the atomistic model is then achieved by controlling the stability gap between two models.
Such stability of dislocation core is still lack of systematic study in the literature. An attempt was made by Hudson and Ortner \cite{Hudson2014-p887-929} for an atomistic model with nearest neighbor interaction. They obtained the stability of a screw dislocation under anti-plane deformation in the sense that the dislocation solution is a global minimizer of the total energy with given total Burgers vector.
To avoid the lattice periodic translation invariant, they fixed the dislocation center.
Although we also fix the center of dislocation, our proofs are quite different from theirs. In particular, we consider both atomistic and continuum models for edge dislocation, and the stabilities are proved in a continuum-to-atomistic way, as shown above.
Again, in the stability analysis of our atomistic model, the infinite-ranged pair potentials lead to an issue in estimating double infinite summations, which is overcome by various summability lemmas obtained in this paper.

There is an extensive literature on the convergence issue of dislocation models using the language of $\Gamma$-convergence \cite{Garroni2006-p535-578,Garroni2010-p1231-1266,Ponsiglione2007-p449-469,DeLuca2012-p1-26,Conti2015-p699-755}.
To the best of our knowledge, they all study the upscaling from the discrete dislocation theory to the dislocation density theory in much larger scales than our situation here. In contrast to these works focusing on many dislocations to dislocation density and neglect the details of the core structure, our work looks into a single dislocation core structure and provide a quantitative error estimate for displacement in the PN dislocation solution with respect to the atomistic dislocation solution. In particular, we obtain the misfit potential in the continuum model from atomistic model according to the exact definition of $\gamma$ surface instead of a quadratic or sinusoidal approximation.

The present paper is organized as follows. We present the derivations of the models and state main results of this paper in Section 2. Section 3 provides some preliminary results for the rest of the analysis. In Section 4, we deal with the consistency issue of the PN model based on asymptotic analysis of the atomistic model. In Section 5, we focus on the existence and stability of the PN model. Section 6 is concerned with the stability of atomistic model. In Section 7, we collect the previous results to prove the existence of the atomistic solution which is close to the continuum solution in the asymptotic sense. Finally, our key assumption on the smallness of $\e$ is validated in the appendix using data based on first principle calculations.

\section{Models and main results}

In this paper, we study the one-dimensional form of the generalized PN model
recently developed for the inter-layer dislocations in bilayer graphene \cite{Dai2016-p85410-85410}. That is,
the dislocation is straight and the structure of the bilayer graphene is uniform in the
direction of the dislocation. We focus on an edge dislocation between a planar bilayer
graphene and neglect the buckling effect \cite{Dai2016-p85410-85410}. This is a reasonably simplified scenario, for instance, when the bilayer is bonded by a substrate such that the buckling is limited. In fact, comparing to in-plane displacement, the out-of-plane displacement affects only slightly the structure of an edge dislocation. As a result, we only study the displacement within the slip plane.
The dislocation solutions are local
minimizers of the total energy in the atomistic model and the PN model, respectively,
subject to the constraint of the total Burgers vector. We will show that the dislocation
solution of the PN model is an approximation of the dislocation solution using the
atomistic model.

\subsection{Atomistic model}\label{sec..atomistic.model}

In the one-dimensional setting, the bilayer graphene consists of two chains of
atoms along the $x$ axis. The two atomic layers are located at $y =\pm\frac{1}{2}d$, respectively,
where $d$ is the distance between two layers. The system is uniform in the $z$ direction.
For a perfect bilayer graphene without dislocation, the atoms are located at $\Gamma_\text{a}^{\pm}=\left\{\textbf{x}^\pm_i=(x_i^\pm,\pm\frac{1}{2}d): i\in \mathbb{Z}\right\}$, where $x_i^+=ia-\frac{1}{2}a$, $x_i^-=ia$,
and  $a$ is the lattice constant, see Fig.~\ref{figure..lattice.dislocation}(a). This perfect lattice is the reference state of the dislocation to be described below.

\begin{figure}[htpb]
\centering
    \subfigure[]{\includegraphics[width=0.65\textwidth]{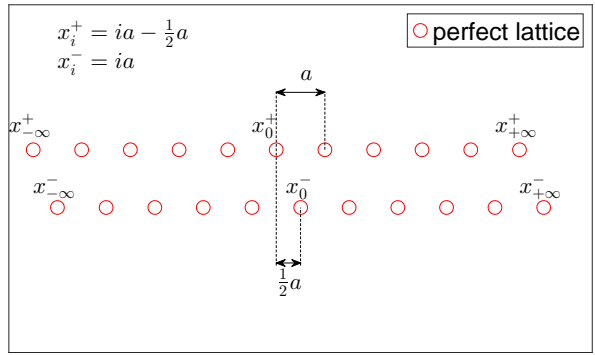}}
    \subfigure[]{\includegraphics[width=0.65\textwidth]{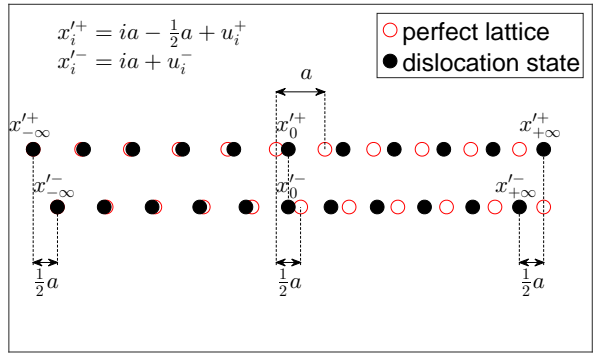}}
    \caption{(a) Perfect lattice. (b) Configuration of an edge dislocation (compared with the reference state).}\label{figure..lattice.dislocation}
\end{figure}

Suppose that there is a dislocation centered at the origin $(0,0)$ with Burgers vector $\textbf{b}=(a,0)$. This dislocation is an edge dislocation.
The dislocation structure is described by using the perfect lattice above as the reference state, and the atomic sites are $\Gamma^{\pm}_{\text{a}}=\{\textbf{x}'^\pm_i=(x_i'^{\pm}, \pm\frac{1}{2}d): i\in \mathbb{Z}\}$, where $x_i'^{+}=x_i^{+}+u_i^{+}=ia-\frac{1}{2}a+u_i^{+}$ and $x_i'^{-}=x_i^{-}+u_i^{-}=ia+u_i^{-}$.
The displacement field $u=\{u^+_i,u^-_i\}_{i\in\mathbb{Z}}$ of this edge dislocation satisfies the boundary conditions at $\pm\infty$:
\begin{eqnarray}
    \lim_{i\rightarrow-\infty}(u_i^+-u_i^-)=0,\, \, \, \lim_{i\rightarrow+\infty}(u_i^+-u_i^-)=a.\label{eq..atom.boundary.condition}
\end{eqnarray}
To fix the center of the dislocation at $(0,0)$, we also assume
\begin{eqnarray}
    u_0^+-u_0^-=a/2.\label{eq..atom.boundary.condition1}
\end{eqnarray}
See the atomic configuration of this dislocation shown in Fig.~\ref{figure..lattice.dislocation}(b).
Here we only consider the displacement within its
own layer, and the vertical
displacement that is normal to the bilayer is neglected due to the non-buckling case.

Suppose that the system is described by pairwise potentials.
The interaction is $V\left(\frac{|\textbf{x}_j'^{\pm}-\textbf{x}_i'^{\pm}|}{a}\right)=V\left(\frac{x_j'^{\pm}-x_i'^{\pm}}{a}\right)$ for atoms $\textbf{x}_j'^{\pm}$ and $\textbf{x}_i'^{\pm}$ in the same layer; while it is $V_{\text{inter}}\left(\frac{|\textbf{x}_j'^{+}-\textbf{x}_i'^{-}|}{a}\right)$ for atoms $\textbf{x}_j'^{+}$ and $\textbf{x}_i'^{-}$ from different layers.
When the distance $d$ between two layers is fixed, we have $|\textbf{x}_j'^{+}-\textbf{x}_i'^{-}|=\sqrt{(x_j'^+-x_i'^-)^2+d^2}$ and the interlayer potential only depends on the horizontal distance $|x_j'^+-x_i'^-|$.
We define
\begin{eqnarray}
    V_d\left(\frac{x_j'^{+}-x_i'^{-}}{a}\right)
    :=V_{\text{inter}}\left(\frac{|\textbf{x}_j'^{+}-\textbf{x}_i'^{-}|}{a}\right)
    =V_{\text{inter}}\left(\frac{\sqrt{(x_j'^+-x_i'^-)^2+d^2}}{a}\right).
\end{eqnarray}
The total energy of the atomistic model is given by
\begin{eqnarray}
    E_\text{a}[u]&=&\frac{1}{2}\sum_{i\in \mathbb{Z}}\sum_{s\in \mathbb{Z}^*}\left\{\left[V\left(\frac{x_{i+s}'^+-x_i'^+}{a}\right)
    -V(s)\right]+\left[V\left(\frac{x_{i+s}'^--x_i'^-}{a}\right)-V(s)\right]\right\}\nonumber\\
    &&+\sum_{i\in\mathbb{Z}}\sum_{s\in \mathbb{Z}}\left[V_d\left(\frac{x_{i+s}'^+-x_i'^-}{a}\right)-V_d\left(s-\frac{1}{2}\right)\right]\nonumber\\
    &=&\frac{1}{2}\sum_{i\in \mathbb{Z}}\sum_{s\in \mathbb{Z}^*}\left[V\left(s+\frac{u^+_{i+s}-u^+_i}{a}\right)+V\left(s+\frac{u^-_{i+s}-u^-_i}{a}\right)-2V(s)\right]\nonumber\\
    &&+\sum_{i\in\mathbb{Z}}\sum_{s\in \mathbb{Z}}\left[V_d\left(s-\frac{1}{2}+\frac{u^+_{i+s}-u^-_i}{a}\right)-V_d\left(s-\frac{1}{2}\right)\right].\label{eq..atom.energy.original}
\end{eqnarray}
Recall that the state of perfect lattice is used as the reference state.

The atomic sites of the edge  dislocation is determined by minimizing the total energy in Eq.~\eqref{eq..atom.energy.original} subject to the displacement conditions in Eqs. \eqref{eq..atom.boundary.condition} and \eqref{eq..atom.boundary.condition1}.

\subsection{Peierls--Nabarro (PN) model}\label{sec..continuum.model}

In the PN model, we consider an edge dislocation with Burgers
vector $\mathbf{b} = (a,0)$ centered at the origin of the $xy$ plane in the bilayer graphene $\Gamma_\text{PN}^+\cup\Gamma_\text{PN}^-$,
where $\Gamma_\text{PN}^{\pm}=\left\{\mathbf{x^\pm}=(x'^\pm,\pm\frac{1}{2}d):x'^\pm=x+u^\pm(x), x\in \mathbb{R}\right\}$
. As in the atomistic model, we only consider the displacement within its
own layer (i.e., the $x$ direction), and call it the horizontal displacement. The vertical
displacement that is normal to the bilayer is neglected. Here $u^+(x)$ and $u^-(x)$ are the
horizontal displacements along the two layers $\Gamma_\text{PN}^+$ and $\Gamma_\text{PN}^-$, respectively.

As in the classical PN model \cite{Peierls1940-p34-37,Nabarro1947-p256-272}, the disregistry (relative displacement) $\phi(x)$ between the two layers is
\begin{eqnarray}
    \phi(x)=u^+(x)-u^-(x).
\end{eqnarray}
The disregistry $\phi(x)$ of this edge dislocation satisfies the boundary conditions
\begin{eqnarray}
    \lim_{x\rightarrow-\infty}\phi(x)=0,\,\, \lim_{x\rightarrow+\infty}\phi(x)=a.\label{eq..cont.boundary.condition}
\end{eqnarray}
We also assume that
\begin{eqnarray}
    \phi(0)=a/2
\end{eqnarray}
to fix the center of the dislocation at $x = 0$. Note that the horizontal displacement
is not continuous in the $y$ direction, and the discontinuity is described by the disregistry function $\phi(x)$. The disregistry function $\phi(x)$ also describes the structure of the
dislocation; more precisely, $\phi′(x)$ is the distribution of the Burgers vector.

In the framework of the PN model \cite{Peierls1940-p34-37,Nabarro1947-p256-272} with $\gamma$-surface \cite{Vitek1968-p773-786}, the total energy of the bilayer system is divided into two parts: an elastic energy  due to the intra-layer elastic interaction and a misfit energy due to the nonlinear interaction between the two layers, which is
\begin{eqnarray}\label{eq..pn.total}
    E_\text{PN}[u]
    =E_\text{elas}[u]+E_\text{mis}[\phi].
\end{eqnarray}
Here $E_\text{elas}[u]$ is the elastic energy  due to the intra-layer elastic interaction in the two layers
\begin{eqnarray}\label{eq..pn.elastic}
    E_\text{elas}[u]
       =\int_{\mathbb{R}} \left(\frac{1}{2}\alpha |\nabla u^+|^2+\frac{1}{2}\alpha|\nabla u^-|^2\right) \D x,
\end{eqnarray}
where $\alpha$ is the elastic modulus. Note that in each layer, the elastic energy density is
$\frac{1}{2} \alpha |\nabla u^{\pm}|^2 $.
 The energy $E_\text{mis}[\phi]$ is the  misfit energy due to the nonlinear interaction between the two layers
\begin{eqnarray}\label{eq..pn.misfit}
    E_\text{mis}[\phi]
    =\int_{\mathbb{R}}\gamma\left(\phi\right)\D x,
\end{eqnarray}
where the density of this misfit energy $\gamma(\phi)$ is the $\gamma$-surface (or the generalized stacking fault energy) \cite{Vitek1968-p773-786} that is defined as the energy increment per unit length when there is a uniform shift of $\phi$ between the two layers. Especially, when $\phi=ia$, $i\in \mathbb{Z}$, the shifted system still has the perfect lattice structure, and $\gamma(\phi)=0$.
In summary, the energy density of the PN model is
\begin{eqnarray}
    W_\text{PN}\left(\phi,\nabla u^+,\nabla u^-\right)=\frac{1}{2}\alpha |\nabla u^+|^2+\frac{1}{2}\alpha |\nabla u^-|^2+\gamma\left(\phi\right).\label{eq..PNdensity}
\end{eqnarray}

The $\gamma$-surface  $\gamma(\phi)$ accounts for the nonlinear interaction between the two layers with displacement discontinuity $\phi$ between them. Using its definition,
the $\gamma$-surface can be calculated from the atomistic model in Sec.~\ref{sec..atomistic.model} by
\begin{eqnarray}\label{eq..gamma}
    \gamma\left(\phi\right)=\frac{1}{a}\sum_{s\in \mathbb{Z}}\left[V_d\left(s-\frac{1}{2}+\frac{\phi}{a}\right)
    -V_d\left(s-\frac{1}{2}\right)\right].
\end{eqnarray}
The constant $\alpha$ in the elastic energy can also be calculated from the atomistic model in Sec.~\ref{sec..atomistic.model} by
\begin{eqnarray}\label{eq..alpha}
    \alpha=\frac{1}{2a}\sum_{s\in \mathbb{Z}^*}V''(s)|s|^2.
\end{eqnarray}
The purpose of this paper is to establish the convergence from the atomistic model in Sec.~\ref{sec..atomistic.model} to the PN model in Eqs.~\eqref{eq..pn.total}--\eqref{eq..pn.misfit}.

This PN model for the bilayer material contains the essential features of the PN models with $\gamma$-surface. That is, the system is considered as two elastic continuums connected by a misfit energy expressed in terms of the $\gamma$-surface that  accounts for the nonlinear interaction between the two elastic continuums. Note that for a dislocation in $\mathbb{R}^3$, as in the classical PN model \cite{Peierls1940-p34-37,Nabarro1947-p256-272} with the $\gamma$-surface \cite{Vitek1968-p773-786} and later generalizations as reviewed in the introduction section, the three-dimensional space is divided by the slip plane of the dislocation into two half-space elastic continuums, and they are connected by a misfit energy expressed in terms of the $\gamma$-surface across the slip plane. The total energy is  $E_\text{PN}=E_\text{elas}+E_\text{mis}$, where
$E_\text{elas}=\int_{\mathbb{R}^3\backslash\{z=0\}}\sum_{i,j=1}^{3}\frac{1}{2}\sigma_{ij}\epsilon_{ij}\D x\D y\D z$ and  $E_\text{mis}=\int_{\mathbb{R}^2}\gamma(\pmb{\phi}(x,y))\D x\D y$. Here the $xy$ plane is the slip plane of the dislocation, and $\frac{1}{2}\sigma_{ij}\epsilon_{ij}$ is the (linear) elastic energy density,  $\sigma_{ij}$ and $\epsilon_{ij}$ are the stress and strain tensors, respectively, and $\gamma(\pmb{\phi})$ is the $\gamma$-surface. Generalization can also be made to replace the energy of linear elasticity in the PN model  by the energy of Cauchy--Born nonlinear elasticity.

\subsection{Weak interlayer interaction and rescaling}\label{sec..rescaling}

For bilayer graphene,  the  van der Waals like
interaction between the two layers is weak compared to the strong
interlayer covalent-bond interaction in each layer \cite{Dai2016-p85410-85410}. That is, $V_d\ll V$ in the atomistic model. We write the relationship as
\begin{eqnarray}\label{eq..V-tildeV}
   V_d=O(\e^2)V,
\end{eqnarray}
where $\e$ is some dimensionless small parameter to be defined below.
Recall that in the PN model for bilayer graphene, the elastic energy $E_\text{elas}$ is due to the interlayer interaction and the misfit energy $E_\text{mis}$ comes from the interaction between the two layers. The dimensionless small parameter $\e$ is defined based on the PN model as follows.

For most part of the system, the atoms are away from the dislocation, and their atomistic structure is close to that of a perfect lattice. For example, when $\phi/a\ll 1$ in the PN model in Sec.~\ref{sec..rescaling}, which happens on the negative part of the $x$ axis away from the origin,  the energy density in the PN model in Eq.~\eqref{eq..PNdensity} is approximated well by a quadratic form:
\begin{eqnarray}
    W_\text{PN}\left(\phi,\nabla u^+,\nabla u^-\right)
    &\approx&\frac{1}{2}\alpha |\nabla u^+|^2+\frac{1}{2}\alpha |\nabla u^-|^2+\frac{1}{2}\gamma''(0)\phi^2\\
    &=&\frac{1}{2}\alpha |\nabla u^+|^2+\frac{1}{2}\alpha |\nabla u^-|^2+\frac{1}{2}a^2\gamma''(0)\left(\frac{\phi}{a}\right)^2.\label{eq..quadratic.form}
\end{eqnarray}
We remark that a similar quadratic form works for the positive part, with the last term in Eq. \eqref{eq..quadratic.form} replaced by $\frac{1}{2}a^2\gamma''(0)\left(\frac{\phi-a}{a}\right)^2$.
The ratio of the coefficients $\frac{a^2\gamma''(0)}{\alpha}$ is a dimensionless constant that characterizes the relative strength of the inter-layer interaction versus the intra-layer interaction.
Recall that the parameter $\alpha$ is expressed in terms of quantities in the atomistic model as in Eq.~\eqref{eq..alpha}. Using the atomistic expression of $\gamma(\phi)$ in Eq.~\eqref{eq..gamma}, we have
\begin{eqnarray}\label{eq..gamma0}
   \gamma''(0)=\frac{1}{a^3}\sum_{s\in \mathbb{Z}}V_d''\left(s-\frac{1}{2}\right).
\end{eqnarray}
Suggested by Eqs. \eqref{eq..alpha}, \eqref{eq..gamma0}, and \eqref{eq..V-tildeV}, we define the dimensionless parameter
\begin{eqnarray}\label{eq..epsilon}
    \e=\textstyle \sqrt{\frac{a^2\gamma''(0)}{\alpha}},
\end{eqnarray}
and assume that
\begin{eqnarray}
    \e\ll 1.
\end{eqnarray}
A validation of this assumption based on values of atomistic and first principles calculations \cite{Dai2016-p85410-85410,Zhou2015-p155438-155438} is given  in the Appendix.

Using  $a/\e$ as the unit length for the spatial variable $x$  and $a$ as the unit length for the displacements in the PN model, we have
the following rescaled quantities:
\begin{eqnarray}\label{eq..rescaling}
    &&\tilde{x}=\frac{\e x}{a},\,\,\tilde{u}^{\pm}=\frac{u^{\pm}}{a},\,\,\tilde{\phi}=\frac{\phi}{a}.
 \end{eqnarray}
  Accordingly, the variables and functionals related to energy densities are rescaled to
\begin{eqnarray}\label{eq..rescaling1}
    &&\tilde{\alpha}=a\alpha, \,\,
    \tilde{\gamma}(\tilde{\phi})=a\gamma(\phi),\\
    &&\tilde{W}_\text{PN}(\tilde{\phi},\nabla_{\tilde{x}} \tilde{u}^+, \nabla_{\tilde{x}}\tilde{u}^-)=\e^{-1}W_\text{PN}(\phi,\nabla u^+, \nabla u^-),\\
    &&\tilde{E}_\text{PN}[u]=\e^{-1}E_\text{PN}[u],\,\,\tilde{E}_\text{a}[u]=\e^{-1}E_\text{a}[u].
 \end{eqnarray}

Using these rescaled variables, the total energy in the PN model can be written as
\begin{eqnarray}
    \tilde{E}_\text{PN}[u]
    &=&\int_{\mathbb{R}} \tilde{W}_\text{PN}(\tilde{\phi},\nabla_{\tilde{x}} \tilde{u}^+, \nabla_{\tilde{x}} \tilde{u}^-) \ \D \tilde{x}\nonumber\\
    &=&\int_{\mathbb{R}}\left\{\frac{1}{2}\tilde{\alpha} |\nabla_{\tilde{x}} \tilde{u}^+|^2+\frac{1}{2}\tilde{\alpha} |\nabla_{\tilde{x}} \tilde{u}^-|^2+\tilde\gamma (\tilde{\phi})\right\} \D \tilde{x},\label{eq..cont.energy.functional}
\end{eqnarray}
where
\begin{eqnarray}
    &&\tilde\alpha=\sum_{s\in \mathbb{Z}^*}\frac{1}{2}V''(s)|s|^2,\label{eq..alpha.rescaled}\\
    &&\tilde{\gamma}(\tilde{\phi})
    =\sum_{s\in \mathbb{Z}} \left[U\left(s-\frac{1}{2}+u^+-u^-\right)-U\left(s-\frac{1}{2}\right)\right].\label{eq..gamma.rescaled}
\end{eqnarray}
Here, following Eq. \eqref{eq..V-tildeV}, we define in the atomistic model that
\begin{eqnarray}\label{eq..newinterP}
    U=\e^{-2}V_d,
\end{eqnarray}
so that $U=O(1)V$.

Finally, using Eq.~\eqref{eq..newinterP}, the total energy in the atomistic model can be written as
\begin{eqnarray}
    \tilde{E}_\text{a}[u]
    &=&\frac{\e^{-1}}{2}\sum_{i\in \mathbb{Z}}\sum_{s\in \mathbb{Z}^*}\left[V\left(s+(\tilde{u}^+_{i+s}-\tilde{u}^+_i)\right)+V\left(s+(u^-_{i+s}-u^-_i)\right)-2V(s)\right]\nonumber\\
    &&+\e\sum_{i\in\mathbb{Z}}\sum_{s\in \mathbb{Z}}[U(s-\frac{1}{2}+(u^+_{i+s}-u^-_i))-{U(s-\frac{1}{2})}].
    \label{eq..atom.energy.functional}
\end{eqnarray}

For simplicity of notations, frow now on, we will still use variables without $\sim$ in the PN model after the above rescaling.

We remark that $E_\text{PN}[u]$ is independent of $\e$, and thence $E_\text{PN}[u]=O(1)$. The first and the second variations of atomistic and continuum models are denoted as
$\delta E_\text{a}[u]$,
$\delta^2 E_\text{a}[u]$,
$\delta E_\text{PN}[u]$, and
$\delta^2 E_\text{PN}[u]$, respectively. Their explicit form are given in Proposition \ref{prop..first.second.variations}.

\subsection{Main results}
For readers' convenience, we first collect assumptions and fix notations. After that, our main results will be stated.

{\bf Assumptions}
Here is the collection of our assumptions which are physically reasonable and will be discussed in details later.
\begin{enumerate}
        \item[A1] (weak inter-layer interaction) $\e\ll 1$.
        \item[A2] (symmetry) $V(x)=V(-x)$ and $U(x)=U(-x)$.
        \item[A3] (regularity) $V\in C^4(\mathbb{R}\backslash\{0\})$ and $U\in C^4(\mathbb{R})$.
        \item[A4] (fast decay)
            $|V^{(4)}(x)|\leq |x|^{-8-\theta}$ and $|U^{(4)}(x)|\leq |x|^{-6-\theta}$, $|x|\geq R$
            for some $R>0$ and $\theta>0$.
        \item[A5] (elasticity constant) $\alpha>0$.
        \item[A6] ($\gamma$-surface)
            $\arg\min_{\phi\in\mathbb{R}}\gamma(\phi)=\mathbb{Z}$ and $\gamma''(0)>0$.
        \item[A7] (small stability gap)
    $
        \kappa>\Delta,
    $
    where
    \begin{eqnarray}
        \Delta&=&\lim_{\e\rightarrow 0}\sup_{\|Df\|_\e=1}\left\langle \delta^2 E_\text{PN}[0]\bar{f},\bar{f} \right\rangle-\left\langle \delta^2 E_\text{a}[0]f,f \right\rangle_{\e},\label{eq..Delta.definition}\\
        \kappa&=&\inf_{\|f\|_{X_0}=1}\langle \delta^2 E_\text{PN}[v]f,f\rangle.
    \end{eqnarray}
     with $v$ being the dislocation solution of the PN model (cf. Theorem \ref{thm..cont.existence.minimizer}). The operators and functional spaces here will be defined in Eqs.~\eqref{eq..difference.operator}--\eqref{eq..inner.product.X_eps}.
\end{enumerate}

We remark that in our bilayer graphene setting, A1--A7 are all satisfied.
In particular, a verification of Assumption A1 is provided in the Appendix, where we show that $\e\approx0.0475\ll 1$ based on the data from Refs. \cite{Dai2016-p85410-85410,Zhou2015-p155438-155438}.

In general, Assumptions A2--A4 are satisfied by most pair potentials, such as the Lennard--Jones potential, the Morse potential, etc. The physical meaning of Assumptions A5--A6 is that the lattice structure without defects is the unique global minimizer of the total energy.

For Assumption A7, we remark that $\Delta\geq 0$ (cf. Proposition \ref{prop..Delta.geq.0}) characterizes the stability gap between atomistic model ($\delta^2 E_\text{a}[0]$) and PN model ($\delta^2 E_\text{PN}[0]$) at perfect lattice, while $\kappa>0$ (only depends on $R,\theta,\alpha$, and $\gamma''(0)$, cf. Proposition \ref{prop..PN.stability}) depicts the stability of the dislocation solution of the PN model. We also provide an explicit formula for $\Delta$ (cf. Proposition \ref{prop..explicit.Delta}). Here are two examples where A7 holds.
\begin{exam}
    [nearest neighbor interaction]
    Let $V$ be nearest neighbor interaction, i.e.,
    $V(s)=0$ for $|s|\geq 2$.
    Then $\Delta=0$ and Assumption A7 holds. (cf. Proposition \ref{prop..NN.Delta}).
\end{exam}
\begin{exam}
    [Lennard--Jones potential]
    Let $V$ be Lennard--Jones $(m,n)$ potential, i.e.,
    \begin{eqnarray}
        V(x)=V_\text{LJ}(x)=-\left(\frac{r_0}{|x|}\right)^{m}+\left(\frac{r_0}{|x|}\right)^{n},\,\,1<m<n,\,\,x\neq 0,
    \end{eqnarray}
    where $r_0$ is some characteristic distance.
    Then $\Delta=0$ and Assumption A7 holds. (cf. Proposition \ref{prop..LJ.Delta}).
\end{exam}

{\bf Notations}
In the proofs, we do not intend to optimize the constants, and hence we frequently use $C$ to be an $\e$-independent constant, which may be different from line to line.

For convenience, we introduce the difference operators $D^\pm_s$ for $f$ defined on $\e\mathbb{Z}$ or $\mathbb{R}$:
\begin{eqnarray}\label{eq..difference.operator}
    D^+_{s}f(x)=\frac{f(x+\e s)-f(x)}{\e},\,\,\,\,D^{-}_{s}f(x)=\frac{f(x)-f(x-\e s)}{\e},\,\,s\in \mathbb{Z}.
\end{eqnarray}
Moreover, we denote $D f = D^+_1 f$ and $D^k f=(D^+_1)^k f$ for $k\in \mathbb{N}$. For function $f$ defined on $\e \mathbb{Z}$, we denote
\begin{eqnarray}
    f_i=f(\e i), \,\,i\in \mathbb{Z}.
\end{eqnarray}

Next, we introduce discrete Sobolev spaces $H^k_\e=H^k_\e(\e\mathbb{Z})=\{f:\|f\|_{\e,k}<\infty\}$, $k\in \mathbb{N}$, where the $H^k_\e$ norm is defined as follows
\begin{eqnarray}
    \|f\|^2_{\e,k}=\e\sum_{0\leq j\leq k}\sum_{i\in \mathbb{Z}}|D^{j} f_i|^2.
\end{eqnarray}
Due to the convention, we denote $L^2_\e=H^0_\e$ with norm $\|\cdot\|_\e=\|\cdot\|_{\e,0}$. We refer the readers to Lemma \ref{lem..space.L2.epsilon} for relations and properties of these spaces. For $f,g\in L^2_{\e}$, their inner products is given by
\begin{eqnarray}
    \left\langle f,g\right\rangle_\e=\e\sum_{i\in \mathbb{Z}}f_i g_i.
\end{eqnarray}
If $f^\pm, g^\pm\in L^2_\e$, then we write $f=(f^+,f^-)\in L^2_\e$, $D^k f=(D^k f^+, D^k f^-)$ and define
\begin{eqnarray}
    \|f\|^2_{\e,k}&=&\|f^+\|_{\e,k}^2+\|f^-\|_{\e,k}^2,\\
    \left\langle f,g\right\rangle_\e&=&\left\langle f^+,g^+\right\rangle_\e+\left\langle f^-,g^-\right\rangle_\e.
\end{eqnarray}
Similarly, if $f^\pm, g^\pm\in L^2$, we write $f=(f^+,f^-)\in L^2$, $\nabla^k f=(\nabla^k f^+, \nabla^k f^-)$ and define
\begin{eqnarray}
    \|f\|^2_{H^k}&=&\|f^+\|_{H^k}^2+\|f^-\|_{H^k}^2,\\
    \left\langle f,g\right\rangle&=&\left\langle f^+,g^+\right\rangle+\left\langle f^-,g^-\right\rangle.
\end{eqnarray}
We simplified the norm $\|\cdot\|_{L^2}$ as $\|\cdot\|$ for $L^2$ functions. The uniform norms on $\e\mathbb{Z}$ is given by
$
    \|f\|_{L^\infty_\e}=\sup_{i\in\mathbb{Z}}|f_i|
$.

If $f=(f^+,f^-)\in L^2_\e$, we define its linear interpolation $\bar{f}=(\bar{f}^+,\bar{f}^-)\in L^2$:
\begin{eqnarray}
    \bar{f}^{\pm}(x)=\frac{(i+1)\e-x}{\e}f^{\pm}_i+\frac{x-i\e}{\e}f^{\pm}_{i+1}\,\,\text{for}\,\,i\e\leq x <(i+1)\e.
\end{eqnarray}
We define the jump of $f=(f^+,f^-)$ in $y$ direction
\begin{eqnarray}
    f^\perp(x)=f^+(x)-f^-(x)\,\,\text{and}\,\,
    f^\perp_i=f^+_i-f^-_i.
\end{eqnarray}
Note that the jump $u^\perp=\phi$ is the disregistry for the displacement of the PN model.
We define solution spaces for our problems as follows
\begin{eqnarray}
    S_0&=&\left\{u=(u^+,u^-)\in H^1_{\text{loc}}:\lim_{x\rightarrow-\infty}u^\perp(x)=0,\lim_{x\rightarrow+\infty}u^\perp(x)=1, u^\pm(0)=\pm\frac{1}{4}\right\},\\
    S_\e&=&\left\{u=(u^+,u^-)\in H^1_{\e,\text{loc}}:\lim_{i\rightarrow-\infty}u^\perp_i=0, \lim_{i\rightarrow+\infty}u^\perp_i=1, u^\pm_0=\pm\frac{1}{4}\right\}.
\end{eqnarray}
Here the functional space $H^1_{\text{loc}}$ (respectively $H^1_{\e,\text{loc}}$) consists of local $H^1$ (respectively $H^1_\e$) functions. Throughout this paper, such evaluations $u^\pm(0)$ are always in the trace sense.

We define the following functional spaces for the analysis of both models
\begin{eqnarray}
    &&X_0=\left\{f=(f^+,f^-)\in H^1_\text{loc}:\|f\|_{X_0}<\infty,f^\pm(0)=0\right\},\\
    &&X_\e=\left\{f=(f^+,f^-)\in H^1_{\e,\text{loc}}:\|f\|_{X_\e}<\infty,f^\pm_0=0\right\},
\end{eqnarray}
where $\|f\|_{X_0}=\left\langle f,f\right\rangle^{1/2}_{X_0}$ and $\|f\|_{X_\e}=\left\langle f,f\right\rangle^{1/2}_{X_\e}$ with the following inner products
\begin{eqnarray}
    \left\langle f,g\right\rangle_{X_0}
    &=&\left\langle \nabla f^+,\nabla g^+\right\rangle +\left\langle \nabla f^-,\nabla g^-\right\rangle+\left\langle f^\perp,g^\perp\right\rangle ,\\
    \left\langle f,g\right\rangle_{X_\e}
    &=&\left\langle Df^+,Dg^+\right\rangle_\e+\left\langle Df^-,Dg^-\right\rangle_\e+\left\langle f^\perp,g^\perp\right\rangle_\e.\label{eq..inner.product.X_eps}
\end{eqnarray}
It is easy to check that $X_0$ and $X_\e$ are both Hilbert spaces with respect to inner products $\langle\cdot,\cdot\rangle_{X_0}$ and $\langle\cdot,\cdot\rangle_{X_\e}$. We remark that
$
    \|f\|_{X_0}^2=\|\nabla f\|^2+\|f^\perp\|^2
$ and
$
    \|f\|_{X_\e}^2=\|D f\|_\e^2+\|f^\perp\|_\e^2
$. Finally, the following linear subspace will also be useful in the proofs
\begin{eqnarray}
    M_\e=\{f=(f^+,f^-)\in H^1_{\e,\text{loc}}:f^+_i=-f^-_i=-f^+_{-i},\,\,i\in\mathbb{Z}\}.\label{eq..def.M_e}
\end{eqnarray}

{\bf Main results}
For the PN model, we solve the minimization problem for $v=(v^+,v^-)\in S_0$:
\begin{eqnarray}
    \inf_{u\in S_0} E_\text{PN}[u].\label{eq..PN.minimization}
\end{eqnarray}
The Euler--Lagrange equation of this minimization problem reads as
\begin{eqnarray}
    \left\{
    \begin{array}{l}
    \delta E_\text{PN}[u]=0,\\
    \lim_{x\rightarrow-\infty}u^\perp(x)=0,\,\, \lim_{x\rightarrow+\infty}u^\perp(x)=1,\,\, u^\pm(0)=\pm\frac{1}{4}.
    \end{array}
    \right.\label{eq..cont.Euler.Lagrange}
\end{eqnarray}
For the atomistic model, we solve the minimization problem for $v^\e=(v^{\e,+},v^{\e,-})\in S_\e$:
\begin{eqnarray}
    \inf_{u\in S_\e} E_\text{a}[u].\label{eq..atom.minimization}
\end{eqnarray}
The Euler--Lagrange equation of this minimization problem reads as
\begin{eqnarray}
    \left\{
    \begin{array}{l}
    \delta E_\text{a}[u]=0,\\
    \lim_{i\rightarrow-\infty}u^\perp_i=0,\,\, \lim_{i\rightarrow+\infty}u^\perp_i=1,\,\, u^\pm_0=\pm\frac{1}{4}.
    \end{array}
    \right.\label{eq..atom.Euler.Lagrange}
\end{eqnarray}

Our main results of this paper are
\begin{thm}[Existence for PN model]\label{thm..cont.existence.minimizer}
    If Assumptions A1--A6 hold, then the PN problem \eqref{eq..cont.Euler.Lagrange} has a unique solution $v=(v^+,v^-)$ and $v\in S_0$ is the $X_0$-global minimizer of the energy functional \eqref{eq..cont.energy.functional}.
    Moreover, $v^+(x)=-v^-(x)$ for all $x\in \mathbb{R}$, and $v^+(\cdot)$ is strictly increasing and smooth (at least $C^5$).
\end{thm}
\begin{thm}[Existence for atomistic model; Convergence]\label{thm..atom.existence.minimizer}
    If Assumptions A1--A7 hold, then there exists an $\e_0$ such that for any $0<\e<\e_0$, the atomistic problem \eqref{eq..atom.Euler.Lagrange} has a solution $v^\e=(v^{\e,+},v^{\e,-})$ and $v^\e\in S_\e$ is a $X_\e$-local minimizer of the energy functional \eqref{eq..atom.energy.functional}.
    Furthermore, $\|v^{\e}-v\|_{X_{\e}}\leq C\e^2$, where $v$ is the dislocation solution of the PN model in Theorem \ref{thm..cont.existence.minimizer}.
\end{thm}
Thanks to the convergence of displacement, we have the following important corollary for convergence of energy.
\begin{cor}[Convergence of energy]\label{cor..energy.consistency}
    If Assumptions A1--A7 hold, then there exists an $\e_0$ such that for any $0<\e<\e_0$ we have
    \begin{eqnarray}
        \left|E_\text{PN}[v]-E_\text{a}[v^\e]\right|\leq C\e^2,
    \end{eqnarray}
    where $v$ and $v^\e$ are the solutions of the PN model and the atomistic model, respectively,
    in Theorems \ref{thm..cont.existence.minimizer} and \ref{thm..atom.existence.minimizer}.
\end{cor}
Note that $E_\text{PN}$ is of order $O(1)$ in this corollary, and hence the relative error is of order $O(\e)$. Before the rescaling, $E_\text{PN}$ is of order $O(\e)$ and the relative error is still of order $O(\e)$.

\section{Preliminaries}
We provide some preliminary results in this section, including the calculation of variations of both models and some lemmata characterizing the properties of pair potentials and $\gamma$-surface.

We first list the explicit expressions of the variations for both model.
\begin{prop}[variations of energies]\label{prop..first.second.variations}
    Suppose that Assumptions A1--A4 hold.

    1. For $u\in S_\e$ and $f, g\in X_\e$, we have
    \begin{eqnarray}
        \left\langle\delta E_{\text{a}}[u],f\right\rangle_\e
        &=&\sum_{i\in\mathbb{Z}}\sum_{s\in\mathbb{Z}^*}\frac{1}{2}\left[V'(s+\e D^+_s u_i^+)(D^+_s f^+_i)
        +V'(s+\e D^+_s u_i^-)(D^+_s f^-_i)
        \right]\nonumber\\
        &&+\e\sum_{i\in\mathbb{Z}}\sum_{s\in\mathbb{Z}}\left[U'(s-\frac{1}{2}+u^+_{i+s}-u^-_i)(f^+_{i+s}-f^-_i)\right]
        ,\label{eq..atom.first.variation}\\
        \left\langle \delta^2 E_\text{a}[u]f,g\right\rangle_\e
        &=&\e\sum_{i\in\mathbb{Z}}\sum_{s\in\mathbb{Z}^*}\frac{1}{2}\left[V''(s+\e D^+_s u^+_i)(D^+_s f^+_i)(D^+_s g^+_i)+V''(s+\e D^+_s u^-_i)(D^+_s f^-_i)(D^+_s g^-_i)\right]\nonumber\\
        &&+\e\sum_{i\in\mathbb{Z}}\sum_{s\in\mathbb{Z}}\left[U''(s-\frac{1}{2}+u^+_{i+s}-u^-_i)(f^+_{i+s}-f^-_i)(g^+_{i+s}-g^-_i)
        \right]
        .\label{eq..atom.second.variation}
    \end{eqnarray}

    2. For $u\in S_0$ and $f, g\in X_0$, we have
    \begin{eqnarray}
        \left\langle \delta E_\text{PN}[u] ,f\right\rangle
        &=&\int_{\mathbb{R}}\left\{
        \alpha \nabla u^+ \nabla f^+ +\alpha \nabla u^- \nabla f^- +\gamma'(u^\perp)f^\perp\right\}\D x,\label{eq..cont.first.variation.weak}\\
        \left\langle \delta^2 E_\text{PN}[u]f,g\right\rangle
        &=&\int_{\mathbb{R}}\left\{
        \alpha \nabla f^+ \nabla g^+ +\alpha \nabla f^- \nabla g^- +\gamma''(u^\perp)f^\perp g^\perp\right\}\D x.\label{eq..cont.second.variation.weak}
    \end{eqnarray}

    Moreover, if $u\in H^2_\text{loc}$, then
    \begin{eqnarray}
        \left(\delta E_\text{PN}[u]\right)^\pm
        &=&-\alpha \nabla^2 u^\pm \pm\gamma'(u^\perp),\label{eq..cont.first.variation.strong}\\
        \left(\delta^2 E_\text{PN}[u]f\right)^\pm
        &=&-\alpha \nabla^2 f^\pm \pm\gamma''(u^\perp)f^\perp.\label{eq..cont.second.variation.strong}
    \end{eqnarray}
\end{prop}
\begin{proof}
Using difference operators, the atomistic energy reads as
\begin{eqnarray*}
    E_\text{a}[u]
    &=&\e^{-1}\sum_{i\in \mathbb{Z}}\sum_{s\in \mathbb{Z}^*}\frac{1}{2}\left[V\left(s+\e D^+_{s}u^+_i\right)+V\left(s+\e D^+_{s}u^-_i\right)-2V(s)\right]\\
    &&+\e\sum_{i\in\mathbb{Z}}\sum_{s\in \mathbb{Z}}\left[U(s-\frac{1}{2}+(u^+_{i+s}-u^-_i))- U(s-\frac{1}{2})\right].
\end{eqnarray*}
Eqs. \eqref{eq..atom.first.variation}--\eqref{eq..cont.second.variation.weak} are obtained via direct calculations.
For $u\in H^2_\text{loc}$, integrating by parts leads to
\begin{eqnarray*}
        \left\langle \delta E_\text{PN}[u] ,f\right\rangle
        &=&\int_{\mathbb{R}}\left\{\left[-\alpha \nabla^2 u^+ +\gamma'(u^\perp)\right]f^+
        +\left[-\alpha \nabla^2 u^- -\gamma'(u^\perp)\right]f^-
        \right\}\D x,\\
        \left\langle \delta^2 E_\text{PN}[u]f,g\right\rangle
        &=&\int_{\mathbb{R}}\left\{\left[-\alpha \nabla^2 f^+ +\gamma''(u^\perp)f^\perp\right]g^+
        +\left[-\alpha \nabla^2 f^- -\gamma''(u^\perp)f^\perp\right]g^-
        \right\}\D x.
    \end{eqnarray*}
Then the fundamental lemma of the calculus of variations implies Eqs. \eqref{eq..cont.first.variation.strong}--\eqref{eq..cont.second.variation.strong}.
\end{proof}

Next, we study the regularity of $\gamma$-surface and summability of pair potentials in our models. For notation economy, we set, for $k=0,1,2,\ldots$
\begin{eqnarray}
    &&V_{k,s}= \esssup_{|\xi-s|\leq\frac{1}{2}|s|}|\nabla^k V(\xi)|, \,\,s\in \mathbb{Z}^*\label{eq..V_s^k}\\
    &&U_{k,s}= \esssup_{|\xi-s+\frac{1}{2}|\leq\frac{1}{2}}|\nabla^k U(\xi)|, \,\,s\in \mathbb{Z},\label{eq..U_s^k}\\
    &&v_{k,s,i}=\esssup_{\e (i-|s|)\leq x \leq \e (i+|s|)}\left|\nabla^k v^+(x)\right|, \,\,i, s\in \mathbb{Z}.\label{eq..u_is^k}
\end{eqnarray}
Roughly speaking, $V_{k,s}$ (or $U_{k,s}$, respectively) is a bound for $\nabla^k V(\xi)$ (or $\nabla^k U(\xi)$, respectively) nearby $\xi=s$, and $v_{k,s,i}$ is a bound for $\nabla v$ in $\e|s|$-neighbor nearby $x=\e i$.
These quantities may appear in proofs from time to time.

\begin{lem}[fast decay and summability]\label{lem..decay.properties.V.U}
    Suppose that Assumptions A3--A4 hold. Then there exists a constant $C=C(R)$ such that
    \begin{eqnarray}
        &&\textstyle |V^{(k)}(x)|\leq C|x|^{-k-4-\theta},\,\,|x|\geq \frac{1}{2},\,\,k=0,1,\ldots,4,\\
        &&\textstyle |U^{(k)}(x)|\leq C|x|^{-k-2-\theta},\,\,|x|>0,\,\,k=0,1,\ldots,4.
    \end{eqnarray}
    Moreover, there exists a constant $C=C(R,\theta)$ satisfying the summability conditions
    \begin{eqnarray}
        &&\textstyle \sum_{s\in \mathbb{Z}^*}|s|^{k+3} V_{k,s}\leq C,\,\,k=0,1,\ldots,4,\\
        &&\textstyle \sum_{s\in \mathbb{Z}}|s|^{k+1} U_{k,s}\leq C,\,\,k=0,1,\ldots,4.
    \end{eqnarray}
\end{lem}
\begin{proof}
    Thanks to Assumption A4, there exists $C=C(R)$ such that $|V^{(4)}(x)|\leq C|x|^{-8-\theta}$ for $|x|\geq \frac{1}{2}$.
    Taking iterative integrals on both sides lead to
    $\textstyle |V^{(k)}(x)|\leq C|x|^{-k-4-\theta}$, $|x|\geq \frac{1}{2}$, $k=0,1,\ldots,4$ for some constant $C$.
    Recall the defintion Eq. \eqref{eq..V_s^k}. Then we have $V_{k,s}\leq C (\frac{1}{2}|s|)^{-k-4-\theta}$. Therefore, for $k=0,1,\ldots,4$
    \begin{eqnarray*}
        \sum_{s\in \mathbb{Z}^*}|s|^{k+3}V_{k,s}
        \leq \sum_{s\in \mathbb{Z}^*}2^{k+4+\theta}C |s|^{-1-\theta}\leq C.
    \end{eqnarray*}
    It is similar to show these properties for $U$.
\end{proof}

\begin{lem}[regularity of $\gamma$-surface]\label{lem..regularity.gamma}
    Suppose that Assumptions A3--A4 hold. Then there exist $C=C(R,\theta)$ and $\e_0=\e_0(R,\theta)$ such that for any $0<\e<\e_0$, we have
    \begin{eqnarray}
        \gamma\in C^4(\mathbb{R})\,\, \text{and}\,\,\|\nabla^k \gamma\|_{L^\infty}\leq C \,\,\text{for}\,\,k=0,1,\cdots,4.
    \end{eqnarray}
\end{lem}
\begin{proof}
    Assumption A3 with Lemma \ref{lem..decay.properties.V.U} implies that $\gamma\in C^4(\mathbb{R})$ and
    \begin{eqnarray}
        \nabla^k\gamma(\xi)=\sum_{s\in \mathbb{Z}}U^{(k)}(s-\frac{1}{2}+\xi),\,\,k=1,2,\cdots,4.
    \end{eqnarray}
    Let $n$ be the nearest integer of $\xi$. By Lemma \ref{lem..decay.properties.V.U} again, we have
    $
        |\nabla^k\gamma(\xi)|\leq \sum_{s\in \mathbb{Z}}U_{k,s+n}\leq C.
    $
    If $k=0$, then $|\gamma(\xi)|\leq \sum_{s\in \mathbb{Z}}\left[U_{0,s+n}+U_{0,s}\right]\leq C$.
\end{proof}
\begin{rmk}
  This regularity of $\gamma$-surface is indispensable and it essentially relies on the regularity and summability of the pair potential $V_d$ (or $U$). Consequently, a smooth dislocation solution depends on the regularity of $V_d$ (or $U$).
\end{rmk}

\begin{lem}[symmetry and local stability of $\gamma$-surface]\label{lem..gamma.properties}
Suppose that Assumptions A1--A6 hold. Then we have the following properties of the $\gamma$-surface
    \begin{eqnarray}\label{eq..symmetry.local.stability.gamma.surface}
        \begin{array}{ll}
            \text{(periodicity)} &\textstyle \gamma(\xi+1)=\gamma(\xi),\,\, \xi\in \mathbb{R},\\
            \text{(symmetry)}&\textstyle \gamma(\xi)=\gamma(-\xi),\,\, \xi\in \mathbb{R},\\
            \text{(local stability)}&\gamma(\xi)\geq \frac{1}{2}\gamma''(0) \xi^2,\,\,|\xi|\leq C,
        \end{array}
    \end{eqnarray}
    where the constant $C=C(R,\theta,\gamma''(0))$.
\end{lem}
\begin{proof}
    Given $\xi\in \mathbb{R}$, let $n$ be the nearest integer of $\xi$. Then the series $\sum_{s\in \mathbb{Z}} U(s-\frac{1}{2}+\xi)$ is absolutely summable and its sum is irrelevant to the summation order. In particular, we have
    $
        \sum_{s\in \mathbb{Z}}[U(s+\frac{1}{2}+\xi)-U(s)]=\sum_{s\in \mathbb{Z}}[U(s-\frac{1}{2}+\xi)-U(s)]
    $.
    That is $\gamma(\xi+1)=\gamma(\xi)$.

    Next, the symmetry $\gamma(\xi)=\gamma(-\xi)$ follows immediately from Assumption A2.

    Finally, Assumption A5 with the regularity of $\gamma$ (Lemma \ref{lem..regularity.gamma}) implies $\gamma(\xi)\geq \frac{1}{2}\gamma''(0) \xi^2$ for $|\xi|\leq C.$
\end{proof}
\begin{rmk}
    In the classic PN model, the misfit energy density reads as $\gamma(\phi)=\frac{\mu b^2}{4\pi^2 d}[1-\cos (2\pi \phi)]$ which satisfies Eq. \eqref{eq..symmetry.local.stability.gamma.surface}. Here $d$ is the interplanar distance, $b$ is the length of the Burgers vector, and $\mu$ is the shear modulus.
\end{rmk}

\section{Existence and Stability of the PN Model}
In this section, we study the dislocation solution of the PN model, in particular, its existence and stability.

For the existence, we rewrite our one-step minimization problem \eqref{eq..PN.minimization} into a two-step minimization problem: first minimizing $u=(u^+,u^-)$ with fixed $u^\perp=\phi$, then minimizing the energy with respect to $\phi$. This two-step procedure becomes a routine since the original works of Peierls and Nabarro \cite{Nabarro1947-p256-272,Peierls1940-p34-37}, however, the equivalence lacks a rigorous proof. Here we provide a detailed discussion on the relation of these two minimization problems. We use our bilayer system setting in order to be consistent with this work. The equivalence result and its proof can both be straightforward extended to the general PN model (e.g., in three dimension and for curved dislocations).

We define the function space for disregistry $\phi$:
\begin{eqnarray*}
    \Phi_0=\left\{\phi\in H^1_{\text{loc}}(\mathbb{R}):\lim_{x\rightarrow-\infty}\phi(x)=0, \lim_{x\rightarrow+\infty}\phi(x)=1, \phi(0)=\frac{1}{2} \right\}.
\end{eqnarray*}
In our bilayer system, the two-step minimization reads as:

(i) given $\phi\in \Phi_0$, find $u_\phi=(u^+_\phi,u^-_\phi)\in S_0$ with $u^\perp_{\phi}=\phi$ such that
\begin{eqnarray}\label{eq..two-step.minimization.inner}
    E_\text{elas}[u_{\phi}]=\inf_{
    u\in S_0,\,\,
    u^\perp=\phi}E_\text{elas}[u],
\end{eqnarray}
and denote $E_\text{elas}^{II}[\phi]=\inf_{
    u\in S_0,\,\,
    u^\perp=\phi}E_\text{elas}[u]$;

(ii) find $\phi^*\in \Phi_0$ such that
\begin{eqnarray}\label{eq..two-step.minimization.outer}
    E_\text{PN}^{II}[\phi^*]=\inf_{\phi\in \Phi_0} E_\text{PN}^{II}[\phi],
\end{eqnarray}
where the total energy functional in this two-step minimization problem is defined as
\begin{eqnarray}
    E_\text{PN}^{II}[\phi]&=&E_\text{elas}^{II}[\phi]
    +E_\text{mis}[\phi].
\end{eqnarray}
We remark that, in general, $E_\text{elas}^{II}[\phi]$ always exists, even if the optimal displacement $u$ may not exist (in $S_0$) for some given disregistry $\phi$ with the consistency $u^\perp=\phi$.
In many applications such as the original PN model, there is an explicit solution for the step (i) problem~\eqref{eq..two-step.minimization.inner}. It follows that one simply needs to solve the step (ii) problem~\eqref{eq..two-step.minimization.outer}. This is a great advantage to use this two-step minimization model.

The following proposition establishes the equivalence between two minimization problems.
\begin{prop}[equivalence between two minimization problems]\label{prop..two.step.minimization}
    Suppose there exist $u^0\in S_0$ such that $E_\text{PN}[u^0]<+\infty$. Then the two-step minimization problem \eqref{eq..two-step.minimization.inner} and \eqref{eq..two-step.minimization.outer} is equivalent to the one-step minimization problem \eqref{eq..PN.minimization} in the following sense:

    1. $m^{I}=m^{II}$, where $m^I=\inf_{u\in S_0}
    E_\text{PN}[u]$ and $m^{II}=\inf_{\phi\in \Phi_0} E_\text{PN}^{II}[\phi]$

    2. Given any minimizing sequence $\{u^i\}_{i=1}^{\infty}$ of problem \eqref{eq..PN.minimization}, then $\{\phi^i:=u^{i,\perp}\}_{i=1}^{\infty}$ is a minimizing sequence of problem
    \eqref{eq..two-step.minimization.outer}. Conversely, given any minimizing sequence $\{\phi^i\}_{i=1}^{\infty}$ of problem \eqref{eq..two-step.minimization.outer}, there exists a sequence $\{u^i\}_{i=1}^{\infty}$ with $u^{i,\perp}=\phi^i$, $i\in \mathbb{N}$ such that $\{u^i\}_{i=1}^{\infty}$ is a minimizing sequence of problem
    \eqref{eq..PN.minimization}.

    3. If $u^*$ is a minimizer of problem \eqref{eq..PN.minimization}, $\phi^*:=u^{*,\perp}$ is a minimizer of problem
    \eqref{eq..two-step.minimization.outer}. Conversely, if $\phi^*$ is a minimizer of problem \eqref{eq..two-step.minimization.outer} and $u^*$ solves
    \begin{eqnarray}
        E_\text{elas}[u^*]=\inf_{u\in S_0,\,\,u^\perp=\phi^*}E_\text{elas}[u],\label{eq..equiv.elastic.I.II}
    \end{eqnarray}
    then $u^*$ is a minimizer of problem \eqref{eq..PN.minimization}.
    In particular, if the minimizer $u^*$ in \eqref{eq..equiv.elastic.I.II} is unique, then $u^*$ and $\phi^*$ has an one-to-one correspondence.
\end{prop}
\begin{rmk}
    Condition \eqref{eq..equiv.elastic.I.II} means $E_\text{elas}[u^{*}]=E_\text{elas}^{II}[\phi^*]$. For most applications, including our case $E_\text{elas}[u]=\int_{\mathbb{R}}\left(\frac{1}{2}\alpha|\nabla u^+|^2+\frac{1}{2}\alpha|\nabla u^-|^2\right)\D x$, the minimizer $u^*\in S_0$ satisfying Eq. \eqref{eq..equiv.elastic.I.II} exists, and it is unique.
\end{rmk}
\begin{proof}
    Although $m^{I}$ and $m^{II}$ may be $-\infty$, they are prevented to be $+\infty$ due to the assumption $E_\text{PN}[u^0]<+\infty$.

    1. If $\{u^i\}_{i=1}^{\infty}$ is a minimizing sequence of problem \eqref{eq..PN.minimization}, then $\lim_{i\rightarrow+\infty}E_\text{PN}[u^i]=m^{I}$. For all $i$,
    \begin{eqnarray}\label{eq..m_II<m_I}
        m^{II}\leq E_\text{PN}^{II}[u^{i,\perp}]
        \leq E_\text{PN}[u^i].
    \end{eqnarray}
    Taking the limit $i\rightarrow +\infty$, we obtain $m^{II}\leq m^{I}$.

    Conversely, if $\{\phi^i\}_{i=1}^{\infty}$ a minimizing sequence of problem \eqref{eq..two-step.minimization.outer}, then
    $
        \lim_{i\rightarrow+\infty}E_\text{PN}^{II}[\phi^i]=m^{II}
    $.
    For any $i$, there exist $u^i\in S_0$ with $u^{i,\perp}=\phi^i$ such that
    $E_\text{elas}[u^i]\leq i^{-1}+E_\text{elas}^{II}[\phi^i]$. Then
    \begin{eqnarray}\label{eq..m_I<m_II}
        m^{I}\leq E_\text{PN}[u^i]
        \leq i^{-1}+E_\text{PN}^{II}[\phi^i]
    \end{eqnarray}
    Taking the limit $i\rightarrow +\infty$, we obtain $m^{I}\leq m^{II}$.
    Hence $m^{I}=m^{II}$.

    2. If $\{u^i\}_{i=1}^{\infty}$ is a minimizing sequence of problem \eqref{eq..PN.minimization}, then we set $\phi^i=u^{i,\perp}$ for all $i\in \mathbb{N}$. Thus $\lim_{i\rightarrow}E_\text{PN}^{II}[\phi^i]=m^{II}$ follows from Eq. \eqref{eq..m_II<m_I} and $m^I=m^{II}$.

    Conversely, if $\{\phi^i\}_{i=1}^{\infty}$ a minimizing sequence of problem \eqref{eq..two-step.minimization.outer}, then we choose $u^i\in S_0$ with $u^{i,\perp}=\phi^i$ such that
    $E_\text{elas}[u^i]\leq i^{-1}+E_\text{elas}^{II}[\phi^i]$.
    Thus $\lim_{i\rightarrow+\infty}E_\text{PN}[u^i]=m^I$ follows from Eq. \eqref{eq..m_I<m_II} and $m^I=m^{II}$.

    3. If $E_\text{PN}[u^*]=m^{I}$, then $E_\text{PN}^{II}[u^{*,\perp}]\leq E_\text{PN}[u^*]=m^I=m^{II}$.
    Conversely, if $E_\text{PN}^{II}[\phi^{*}]=m^{II}$ and
    $E_\text{elas}[u^{*}]=\inf_{u\in S_0,\,\,u^\perp=\phi^*}E_\text{elas}[u]$, then
    \begin{eqnarray*}
        E_\text{PN}[u^*]
        =E_\text{elas}[u^*]+E_\text{mis}[u^{*,\perp}]
        =E_\text{elas}^{II}[\phi^*]+E_\text{mis}[\phi^*]=E_\text{PN}^{II}[\phi^*]=m^{I}.
    \end{eqnarray*}
\end{proof}

Now we prove Theorem \ref{thm..cont.existence.minimizer} by solving the two-step minimization. The first step is explicitly solvable. Next, the existence of the minimizer $\phi$ is then proved by the direct method in the calculus of variations. Different from the standard case, any admissible function $\phi\in S_0$ is definitely not $L^2$. Hence, a reference state $\phi^0$ is needed. We then finish the proof by working on the deviation of the solution $\phi-\phi^0$.

\begin{proof}[Proof of Theorem \ref{thm..cont.existence.minimizer}]
    1. Two-step minimization problem. Recall that
    $
        E_\text{elas}[u]=\int_{\mathbb{R}}\left(\frac{1}{2}\alpha|\nabla u^+|^2+\frac{1}{2}\alpha|\nabla u^-|^2\right)\D x.
    $
    For any $\phi\in\Phi_0$, we have
    \begin{eqnarray*}
    \arg\min_{u\in S_0,u^\perp=\phi}E_\text{elas}[u]
    =\arg\min_{u\in S_0}\int_{\mathbb{R}}\left(\frac{1}{2}\alpha|\nabla u^+|^2+\frac{1}{2}\alpha|\nabla u^+-\nabla \phi|^2\right)\D x
    =\left(\frac{1}{2}\phi,-\frac{1}{2}\phi\right).
    \end{eqnarray*}
    Moreover,
    $E_\text{elas}^{II}[\phi]=E_\text{elas}[(\frac{1}{2}\phi,-\frac{1}{2}\phi)]=\frac{1}{4}
    \int_{\mathbb{R}}\alpha|\nabla \phi|^2\D x$. By Proposition \ref{prop..two.step.minimization}, we only need to minimize the following energy $E_\text{PN}^{II}[\phi]$ in terms of disregistry $\phi$:
    \begin{eqnarray}
        E_\text{PN}^{II}[\phi]=\int_{\mathbb{R}}\left(\frac{1}{4}\alpha|\nabla \phi|^2+\gamma(\phi)\right)\D x.\label{eq..E_c,II}
    \end{eqnarray}

    2. Existence. Let $\phi^0(x)=\min\{\max\{x+\frac{1}{2},0\},1\}$ for $x\in \mathbb{R}$. Denote $m=\inf_{\phi\in \Phi_0} E_\text{PN}^{II}[\phi]$. By Assumption A5, $\inf_{\xi\in\mathbb{R}}\gamma(\xi)=\gamma(0)=0$, and hence $m\geq 0$. Also $m\leq E_\text{PN}^{II}[\phi^0]<+\infty$. Hence $m$ is finite. Let $\{\phi^k\}_{k=1}^{\infty}\subset \Phi_0$ be a minimizing sequence for $E_\text{PN}^{II}[\cdot]$.

    Let $\omega^k=\phi^k-\phi^0$. Then $\|\nabla \omega^k\|^2\leq \|\nabla \phi^k\|^2+\|\nabla \phi^0\|^2\leq \frac{4}{\alpha}E_\text{PN}^{II}[\phi^k]+\frac{4}{\alpha}E_\text{PN}^{II}[\phi^0]$. Next we estimate $\|\omega^k\|^2$. According to Lemma \ref{lem..gamma.properties}, there exist a constant $c_0 (\leq \frac{1}{4})$ such that $\gamma(\xi)\geq \frac{1}{2}\gamma''(0) \xi^2$ for $|\xi|\leq c_0$. Let $m'=\min_{c_0\leq \xi \leq 1-c_0}\gamma(\xi)>0$. Note that the characteristic function $\chi_{\{c_0\leq \phi^k(x) \leq 1-c_0\}}\leq \frac{\gamma(\phi^k)}{m'}$. We have
    \begin{eqnarray*}
        \|\omega^k\|^2
        &=&\textstyle \int_{-\frac{1}{2}}^{\frac{1}{2}}|\phi^k-x-\frac{1}{2}|^2\D x+\int_{-\infty}^{-\frac{1}{2}}|\phi^k-0|^2\D x+\int_{\frac{1}{2}}^{+\infty}|\phi^k-1|^2\D x\\
        &\leq&\textstyle 1+\int_{-\infty}^{-\frac{1}{2}}|\phi^k|^2\chi_{\{0\leq \phi^k\leq c_0\}}\D x+\int_{-\infty}^{-\frac{1}{2}}|\phi^k|^2\chi_{\{c_0\leq \phi^k\leq 1\}}\D x\\
        &&+\textstyle \int_{\frac{1}{2}}^{+\infty}|\phi^k-1|^2\chi_{\{1-c_0\leq \phi^k\leq 1\}}\D x+\int_{\frac{1}{2}}^{+\infty}|\phi^k-1|^2\chi_{\{\frac{1}{2}\leq \phi^k\leq 1-c_0\}}\D x\\
        &\leq&\textstyle 1+\int_{-\infty}^{-\frac{1}{2}}\frac{2}{\gamma''(0)}\gamma(\phi^k)\D x+\int_{-\infty}^{-\frac{1}{2}}\frac{1}{m'}\gamma(\phi^k)\D x+\int_{\frac{1}{2}}^{+\infty}\frac{2}{\gamma''(0)}\gamma(\phi^k)\D x+\int_{\frac{1}{2}}^{+\infty} \frac{1}{m'}\gamma(\phi^k)\D x\\
        &\leq&\textstyle 1+(\frac{2}{\gamma''(0)}+\frac{1}{m'})E_\text{PN}^{II}[\phi^k].
    \end{eqnarray*}
    Therefore
    \begin{eqnarray*}
        \|\omega^k\|_{H^1}^2
        =\|\omega^k\|^2+\|\nabla \omega^k\|^2
        \leq \textstyle 1+(\frac{2}{\gamma''(0)}+\frac{1}{m'}+\frac{4}{\alpha})E_\text{PN}^{II}[\phi^k]+\frac{4}{\alpha}E_\text{PN}^{II}[\psi^0].
    \end{eqnarray*}
    Since $\phi^k$ is a minimizing sequence, we obtain that $\omega^k$ is uniformly bounded in $H^1$.

    Passing to a subsequence, $\omega^k$ converges weakly to $\omega^*$ in $H^1$. The Sobolev imbedding theorem implies $\{\omega^k\}_{k=1}^\infty\subset C^{0,\frac{1}{2}}(\mathbb{R})$ and $\omega^*\in C^{0,\frac{1}{2}}(\mathbb{R})$. This leads to $\lim_{x\rightarrow\pm\infty}\omega(x)=0$ and $\omega(x)=0$. Let $\phi^*=\omega^*+\phi^0$. Then $\phi^*\in \Phi_0$ and $\phi^k-\phi^*$ converges weakly to $0$ in $H^1(\mathbb{R})$. Thanks to the convexity of $E_\text{elas}^{II}[\phi]$, $E_\text{PN}^{II}[\cdot]$ is weakly lower semicontinuous on $H^1_{\text{loc}}(\mathbb{R})$. Thus $E_\text{PN}^{II}[\phi^*]\leq \lim\inf_{k\rightarrow\infty}E_\text{PN}^{II}[\phi^k]=m$. It follows that $E_\text{PN}^{II}[\phi^*]=m=\min_{\phi\in \Phi_0}E_\text{PN}^{II}[\phi]$. Thus $\phi^*$ is the minimizer of energy functional \eqref{eq..E_c,II} in $\Phi_0$. Therefore, $v:=(\frac{1}{2}\phi^*,-\frac{1}{2}\phi^*)$ is the minimizer of energy functional $E_\text{PN}[\cdot]$ in $S_0$.

    3. Euler--Lagrange equation. Since $\phi^*$ is the minimizer of energy functional $E_\text{PN}^{II}[\cdot]$ in $\Phi_0$, it is the weak solution of the Euler--Lagrange equation
    \begin{eqnarray}
        -\frac{1}{2}\alpha \nabla^2 \phi^*-\gamma'(\phi^*)=0.\label{eq..E-L.phi}
    \end{eqnarray}
    Notice that $\gamma'(\cdot)$ is continuous. As a result, $\phi^*\in C^2$ is a classical solution of the Euler--Lagrange equation \eqref{eq..E-L.phi}. Therefore $v=(\frac{1}{2}\phi^*,-\frac{1}{2}\phi^*)\in C^2$ is the classical solution of the Euler--Lagrange equations \eqref{eq..cont.Euler.Lagrange}.

    4. Monotonicity. For $x\in \mathbb{R}$, we have $-\frac{1}{2}\alpha (\nabla^2 \phi^*)(\nabla \phi^*)+\gamma'(\phi^*)(\nabla \phi^*)=0$. Taking integral from $x$ to $+\infty$, we have
    $
        -\frac{1}{4}\alpha (\nabla \phi^*(x))^2+\gamma(\phi^*(x))
        =\lim_{\xi\rightarrow+\infty}\left[-\frac{1}{4}\alpha(\nabla \phi^*(\xi))^2+\gamma(\phi^*(\xi))\right]=0
    $.
    Thus $\nabla \phi^*=\pm\sqrt{\frac{4}{\alpha}\gamma(\phi^*)}$. Since $\nabla \phi^*$ is continuous and $\frac{4}{\alpha}\gamma(\phi^*)>0$ for all $x\in\mathbb{R}$, $\nabla \phi^*$ does not change the sign. Hence $\nabla \phi^*=\sqrt{\frac{4}{\alpha}\gamma(\phi^*)}>0$ follows the fact that $\int_{-\infty}^{+\infty} \nabla \phi^*\D x=1>0$. Therefore, $\nabla v^{\pm}=\pm\sqrt{\frac{1}{\alpha}\gamma(\pm 2v^\pm)}$. In other words, $v^+$ (respectively, $v^-$) is monotonically increasing (respectively, decreasing) on $\mathbb{R}$.

    5. Uniqueness. The uniqueness of the classical solution of the Euler--Lagrange equations \eqref{eq..cont.Euler.Lagrange} follows from that of the initial value problem $\nabla v^{\pm}=\pm\sqrt{\frac{1}{\alpha}\gamma(\pm 2v^\pm)}$ with the initial condition $v^\pm(0)=\pm\frac{1}{4}$.

    6. Symmetry. Since $v=(\frac{1}{2}\phi^*,-\frac{1}{2}\phi^*)$ is the unique minimizer of the $E_\text{PN}[\cdot]$ in $S_0$, we immediately have the symmetry $v^+(x)=-v^-(x)$ for all $x\in \mathbb{R}$.

    7. Regularity. Note that $\|\phi^*\|_{L^\infty}\leq 1$. Since $\nabla \phi^*\in C^1(\mathbb{R})$, $\nabla \phi^*\geq 0$ and the fact that $\phi^*$ is bounded, we have $\lim_{x\rightarrow\pm\infty}\nabla \phi^*(x)=0$. Thus $\|\nabla \phi^*\|_{L^\infty}\leq C$.
    Utilizing Eq. \eqref{eq..E-L.phi}, it is no difficulty to bootstrap the regularity of $\phi^*$, and hence the regularity of $v=(\frac{1}{2}\phi^*,-\frac{1}{2}\phi^*)$. Indeed, thanks to Lemma \ref{lem..regularity.gamma}, we have $v\in C^k(\mathbb{R})$ and $\|\nabla^k v\|_{L^\infty}\leq C$ for $k=3,4,5$, where $C=C(\alpha,R,\theta)$ is independent of $\e$.
\end{proof}

A corollary of Theorem \ref{thm..cont.existence.minimizer} shows the symmetry property of $v^\pm$.
\begin{cor}
    Let $v=(v^+,v^-)$ be the dislocation solution of the PN model in Theorem \ref{thm..cont.existence.minimizer}. Then $v$ has the symmetry with respect to $x$: $v^+(x)+v^+(-x)=\frac{1}{2}$ and $v^-(x)+v^-(-x)=-\frac{1}{2}$, $x\in \mathbb{R}$.
\end{cor}
\begin{proof}
    By the symmetry and periodicity of $\gamma$-surface (cf. Lemma \ref{lem..gamma.properties}), we have $\gamma(\frac{1}{2}+\xi)=\gamma(\xi-\frac{1}{2})=\gamma(\frac{1}{2}-\xi)$ for all $\xi\in \mathbb{R}$. Then it is easy to see the solution of ODE $\nabla \phi^*=\sqrt{\frac{4}{\alpha}\gamma(\phi^*)}$ with initial value $\phi^*(0)=\frac{1}{2}$ satisties
    $\phi^*(x)-\frac{1}{2}=\frac{1}{2}-\phi^*(-x)$ for $x\geq 0$. This with the fact that $v=(\frac{1}{2}\phi^*,-\frac{1}{2}\phi^*)$ completes the proof.
\end{proof}

Due to the translation invariant, the second variation of energy at the dislocation solution $\delta^2 E_\text{PN}[v]$ has a zero eigenvalue. The following proposition guarantees that this zero eigenvalue is simple. In other words, the eigenfunctions corresponding to zero eigenvalue form a one-dimension linear space.
\begin{prop}[zero eigenvalue is simple]\label{prop..cont.zero.eigenvalue}
    Suppose that Assumptions A1--A6 hold. Let $v$ be the dislocation solution of the PN model in Theorem \ref{thm..cont.existence.minimizer}. If $f\in C^2$ with $\|f\|_{X_0}<\infty$ and $f$ solves $\delta^2 E_\text{PN}[v]f=0$, then $f=A \nabla v+B$ for some constants $A$ and $B$.
\end{prop}

\begin{proof}
    Let $g=\nabla v$.
    Thus we have
    \begin{eqnarray*}
        &&\left(\delta^2 E_\text{PN}[v]f\right)^\pm=-\alpha \nabla^2 f^\pm \pm\gamma''(v^\perp)(f^+-f^-)=0,\\
        &&\left(\delta^2 E_\text{PN}[v]g\right)^\pm=-\alpha \nabla^2 g^\pm \pm\gamma''(v^\perp)(g^+-g^-)=\nabla \left[-\alpha \nabla^2 v^\pm \pm\gamma'(v^\perp)\right]=0.
    \end{eqnarray*}
    The first equation implies $\nabla^2 f^+ (x)=-\nabla^2 f^-(x)$ for all $x\in \mathbb{R}$. Thus $\nabla f^+(x)+\nabla f^-(x)$ is a constant for all $x\in \mathbb{R}$. Since $f\in C^2$ and $\|f\|_{X_0}<\infty$, we have $\nabla f^+(x)=-\nabla f^-(x)$ for all $x\in \mathbb{R}$. Thus $f^+(x)=-f^-(x)+2B$ for some constant $B$ and all $x\in \mathbb{R}$. Let $h(x)=f(x)-B$. Then $f^+-f^-=2f^+-2B=2h^+=-2h^-$. Note that $g^+(x)=-g^-(x)$ for all $x\in \mathbb{R}$. Then we have
    \begin{eqnarray}
        &&-\alpha \nabla^2 h^\pm+2\gamma''(v^\perp)h^\pm=0,\\
        &&-\alpha \nabla^2 g^\pm+2\gamma''(v^\perp)g^\pm=0.
    \end{eqnarray}
    Eliminating $\gamma''(v^\perp)$ term leads to
    \begin{eqnarray*}
        -\alpha g^\pm\nabla^2 h^\pm+\alpha h^\pm \nabla^2 g^\pm=0\,\,\text{or}\,\,\alpha \nabla \left(g^\pm \nabla h^\pm-h^\pm\nabla g^\pm\right)=0.
    \end{eqnarray*}
    Thus $g^\pm\nabla h^\pm-h^\pm\nabla g^\pm$ is a constant.
    Since $f\in C^2$ and $\|f\|_{X_0}<\infty$, we can derive that $h^\pm, \nabla h^\pm \rightarrow 0$ as $|x|\rightarrow \infty$. Hence $g^\pm\nabla h^\pm-h^\pm\nabla g^\pm=0$ for all $x\in \mathbb{R}$. By strictly monotonicity of $v^\pm$ (cf. Theorem \ref{thm..cont.existence.minimizer}), we have $g^\pm=\nabla v^\pm \neq 0$. Thus
    $(g^\pm)^2\nabla \left(\frac{h^\pm}{g^\pm}\right)=g^\pm\nabla h^\pm-h^\pm \nabla g^\pm=0$. Therefore $h=A g=A \nabla v$ and $f=A \nabla v +B$ for some constants $A$ and $B$.
\end{proof}

\begin{rmk}
    The physical meaning of Proposition \ref{prop..cont.zero.eigenvalue} is that the dislocation solution $v$, satisfying the boundary conditions but not the center condition, is invariant under translation. Indeed, let us consider an infinitesimal translation $\D x$ of the dislocation solution. The translated displacement field is $v(x+\D x)$ and hence the perturbation is $v(x+\D x)-v(x)=(\nabla v)\D x$. This perturbation mode is exactly the eigenfunction, in the previous proposition, corresponding to the zero eigenvalue.
\end{rmk}

Now we are ready to obtain the stability result of the PN model. Later, we will see that the stability of the atomistic model can be achieved by this PN stability with the small stability gap Assumption A7.
\begin{prop}[stability of PN model]\label{prop..PN.stability}
    Suppose that Assumptions A1--A6 hold. Let $v$ be the dislocation solution of the PN model in Theorem \ref{thm..cont.existence.minimizer}.
    There exists a constant $\kappa=\kappa(R,\theta,\alpha,\gamma''(0))>0$ such that for $f\in X_0$, we have
    \begin{eqnarray}
       \left\langle \delta^2 E_\text{PN}[v]f, f\right\rangle\geq \kappa\|f\|_{X_0}^2.
    \end{eqnarray}
\end{prop}

\begin{proof}
    We prove the statement by contradiction.
    Suppose there exists a sequence $\left\{f^n\right\}_{n=1}^{\infty}$ satisfying the following conditions:
    \begin{eqnarray}
        \|f^n\|_{X_0}=1\,\,\,\,\text{and}\,\,\,\,
        \textstyle \frac{1}{n}\|f^n\|_{X_0}^2>\left\langle \delta^2 E_\text{PN}[v]f^n,f^n\right\rangle
        =\textstyle I[f^n],
    \end{eqnarray}
    where the functional
    $
        I[f]=\int_{\mathbb{R}}\left\{\alpha |\nabla f^{+}|^2+\alpha |\nabla f^{-}|^2+\gamma''(v^\perp)(f^{\perp})^2\right\}\D x
    $.

    The uniformly boundedness $\|f^n\|_{X_0}=1$ implies that there exists a subsequence $\left\{f^{k_n}\right\}_{n=1}^{\infty}$ with $f^*\in X_0$ satisfying
    (1) $f^{k_n,\pm}_x\rightarrow f^{*,\pm}_x$ weakly in $L^2$ and (2) $f^{k_n,\perp}\rightarrow f^{*,\perp}$ strongly in $L^2$.
    By lower semi-continuity, we have $I[f^*]\leq 0$.
    Since $v$ minimizes the energy $E_\text{c}$, we have $I[f^*]\geq 0$. Thus $f^*$ minimizes the functional $I[f]$
    and hence solves Euler--Lagrange equation in the weak sense
    \begin{eqnarray*}
        -\alpha \nabla^2 f^{*,\pm}\pm\gamma''(v^\perp)f^{*,\perp}=0.
    \end{eqnarray*}
    Note that $\gamma''(v^\perp)$ is continuous by Lemma \ref{lem..regularity.gamma}. We apply the Schauder estimate and obtain $f^{*,\pm}\in C^{2,\alpha}_{\text{loc}}(\mathbb{R})$ \cite{Gilbarg2001-p-}. Proposition \ref{prop..cont.zero.eigenvalue} implies $f^*=A \nabla v+B$. Note that $A \nabla v^\perp(0)=f^{*,\perp}(0)=0$ and $\nabla v^\perp(0)\neq 0$. Then $A=0$ and $f^{*,\pm}\equiv B$ for some constant $B\in\mathbb{R}$.
    There exists $K<\infty$, such that $\gamma''(v^\perp(x))\geq \frac{1}{2}\gamma''(0)>0$ on $\mathbb{R}\backslash(-K,K)$. Notice that $H^1(\mathbb{R})$ can be embedded in $C^{0,\frac{1}{2}}(\mathbb{R})$. Utilizing Arzela--Ascoli theorem, we obtain $f^{k_n,\perp}\rightarrow f^{*,\perp}\equiv 0$ uniformly on $(-K,K)$. Therefore
    \begin{eqnarray*}
        \lim_{n\rightarrow\infty}I[f^n]
        &\geq&-\sup_{x\in \mathbb{R}}|\gamma''(v^\perp(x))|\lim_{n\rightarrow\infty}\int_{-K}^{K}(f^{n,\perp})^2\D x\\
        &&+\alpha\lim_{n\rightarrow\infty}\int_{\mathbb{R}}\left\{|\nabla f^{n,+}|^2+|\nabla f^{n,-}|^2\right\}\D x+\lim_{n\rightarrow\infty}\int_{\mathbb{R}\backslash(-K,K)}\gamma''(v^\perp)(f^{n,\perp})^2\D x\\
        &\geq&\min\left\{\alpha,\frac{1}{2}\gamma''(0)\right\}\lim_{n\rightarrow\infty}\left\{\int_{\mathbb{R}}\left( |\nabla f^{n,+}|^2+|\nabla f^{n,-}|^2\right)\D x+\int_{\mathbb{R}\backslash(-K,K)}(f^{n,\perp})^2\D x\right\}\\
        &=&\min\left\{\alpha,\frac{1}{2}\gamma''(0)\right\}>0.
    \end{eqnarray*}
    This contradicts with $\lim_{n\rightarrow\infty}I[f^n] \leq\lim_{n\rightarrow\infty}\frac{1}{n}\|f^n\|_{X_0}^2=0$. Hence the original statement holds.
\end{proof}

\section{Consistency of the PN Model}
In this section, the force consistency is obtained at the dislocation solution of the PN model. More precisely, the force in the atomistic model is $O(\e^2)$-close to its counterpart in the PN model, provided that the displacement of the atomistic model is exactly the dislocation solution in Theorem \ref{thm..cont.existence.minimizer}. This asymptotic analysis is not only formal but also rigorous in the sense that we estimate the truncation error in $X_\e$ norm.

Here we first provide several lemmata connecting the discrete Sobolev spaces.
\begin{lem}[property of discrete Sobolev norms]\label{lem..space.L2.epsilon}
    For $k\in \mathbb{N}$, we have
    \begin{eqnarray}
        \|f\|_\e\leq \|f\|_{\e,k}\leq 2^{k+1}\max\{1,\e^{-k}\}\|f\|_\e.
    \end{eqnarray}
\end{lem}
\begin{proof}
    By definition, we have $\|f\|_\e^2\leq \|f\|_{\e,k}^2$ and $\|D^{j} f\|_\e^2\leq 4\e^{-2}\|D^{j-1} f\|_\e^2\leq 2^{2j}\e^{-2j}\|f\|_\e^2$ for $j=1,\cdots,k$. Then
    $
        \|f\|_{\e,k}^2
        \leq\sum_{j=0}^{k}2^{2j}\e^{-2j}\|f\|_\e^2
        \leq 2^{2k+2}\max\{1,\e^{-2k}\}\|f\|_\e^2
    $.
\end{proof}

\begin{lem}[property of $M_\e$]\label{lem..subspace.M}
The linear space $M_\e$ is a Hilbert space with inner product $\langle\cdot,\cdot\rangle_{\e}$. Moreover, we have $M_\e\subset H^1_\e$ and
    \begin{eqnarray}
        \|f\|^2_{\e,1}\leq \|f\|^2_{X_\e}\leq 2\|f\|^2_{\e,1}.\label{eq..M_e.norm}
    \end{eqnarray}
\end{lem}
\begin{proof}
    The Hilbert space is easy to check. And Eq. \eqref{eq..M_e.norm} follows from $\|f^\perp\|^2_\e=2\|f\|^2_\e$ for $f\in M_\e$.
\end{proof}

\begin{lem}[property of finite difference operator $D^\pm_s$]\label{lem..sum.D_s.f.L2}
    If $s\in \mathbb{Z}^*$ and $f\in L^2_\e$, then
    \begin{eqnarray}
        \|D^\pm_s f\|_\e\leq |s|\|Df\|_\e.
    \end{eqnarray}
\end{lem}
\begin{proof}
    Without loss of generality, we suppose $s>0$ and prove the result for $D^+_s f$.
    By the Cauchy--Schwarz inequality, we have
    $
        (D_s^+ f^\pm_i)^2=(\sum_{j=i}^{i+s-1}D f^\pm_j)^2\leq s\sum_{j=i}^{i+s-1} (Df^\pm_j)^2
    $.
    Then $\|D^+_s f\|^2_\e\leq s^2\|D f\|_\e^2$ follows this.
\end{proof}

The following summability lemma is quite helpful in estimating the truncation errors (cf. Proposition \ref{prop..consistency}).
\begin{lem}[summability of $v$]\label{lem..u^k_is.L2norm}
    Let $v$ be the dislocation solution of the PN model in Theorem \ref{thm..cont.existence.minimizer}. Given $k=0,1,\cdots,4$ and $s\in \mathbb{Z}^*$, $\e\leq 1$, we have
    \begin{eqnarray}
        \e\sum_{i\in\mathbb{Z}}v_{k,s,i}\leq C|s|\,\,\text{and}\,\,
        \|v_{k,s}\|_\e^2\leq C |s|,\label{eq..u_is^k_L2norm}
    \end{eqnarray}
    where $C=C(\|\nabla v\|_{W^{k,1}},\|v\|_{W^{k,\infty}})$ is independent of $s$.
\end{lem}
\begin{proof}
    Without loss of generality, we suppose that $s>0$.
    For each $i\in \mathbb{Z}$, there exists some $\xi_i$ with $\e(i-s)\leq \xi_i\leq\e(i+s)$ satisfying $v_{k,s,i}=|\nabla^k v^+(\xi_i)|$. Note that
    $
        \sum_{i\in \mathbb{Z}}v_{k,s,i}=\sum_{j=0}^{2s-1}\sum_{n\in \mathbb{Z}}v_{k,s,2ns+j}.
    $
    Then for each $j\in \{0,1,2,\cdots, 2s-1\}$, we have
    \begin{eqnarray*}
        2s\e\sum_{n\in \mathbb{Z}}v_{k,s,2ns+j}
        &\leq&  \sum_{n\in \mathbb{Z}}\int_{\e (2(n-1)s+j)}^{\e (2ns+j)}|\nabla^k v^+(\xi_{2ns+j})-\nabla^k v^+(x)|+|\nabla^k v^+(x)|\D x\\
         &\leq&  \sum_{n\in \mathbb{Z}}\int_{\e (2(n-1)s+j)}^{\e (2ns+j)}
        \left(\int_{x}^{\xi_{2ns+j}}|\nabla^{k+1}v^+(\xi)|\D \xi\right) \D x+\|\nabla^k v^+\|_{L^1}\\
         &\leq& 2s\e\|\nabla^{k+1} v^+\|_{L^1}+\|\nabla^k v^+\|_{L^1}.
    \end{eqnarray*}
    Hence
    $
        \e\sum_{i\in \mathbb{Z}} v_{k,s,i}\leq 2s\e \|\nabla^{k+1}v^+\|_{L^1}+\|\nabla^{k}v^+\|_{L^{1}}\leq2s\|\nabla v^+\|_{W^{k,1}}
    $. Obviously, we have $\esssup_{i\in \mathbb{Z}} v_{k,s,i}\leq \|v^+\|_{W^{k,\infty}}$.
    Eq. \eqref{eq..u_is^k_L2norm} follows this.
\end{proof}

\begin{prop}[consistency of PN model]\label{prop..consistency}
    Suppose that Assumptions A1--A6 hold. Let $v$ be the dislocation solution of the PN model in Theorem \ref{thm..cont.existence.minimizer}, then there exist $C$ and $\e_0$ such that for $0<\e< \e_0$ and $f\in M_\e$ we have
    \begin{eqnarray}\label{eq..consistence.Lp}
        |\langle\delta E_\text{a}[v]-\delta E_\text{PN}[v],f\rangle_\e|\leq C\e^2\|f\|_{X_\e}
    \end{eqnarray}
    Here $C$ and $\e_0$ depend on $R$, $\theta$, $\alpha$, and $\gamma''(0)$.
\end{prop}

\begin{proof}
    The Sobolev imbedding theorem says that $\|\nabla v\|_{W^{3,\infty}}\leq C\|\nabla v\|_{W^{4,1}}$.
    By Theorem \ref{thm..cont.existence.minimizer}, we have $\|v\|_{W^{4,\infty}}\leq \max\{\|v\|_{L^\infty}, C\|\nabla v\|_{W^{4,1}}\}$. Let $\e\leq 1$.

    1. We rewrite the difference $
        \langle \delta E_\text{a}[v]-\delta E_\text{PN}[v],f
        \rangle_\e
        =R_\text{elas}+R_\text{mis}$,
    where
    \begin{eqnarray*}
        R_\text{elas}&=&-\sum_{i\in\mathbb{Z}}\sum_{s\in\mathbb{Z}^*}\sum_\pm\frac{1}{2}\left\{D^-_s[V'(s+\e D^+_s v_i^\pm)]-\e V''(s) s^2 \nabla^2 v^\pm_i\right\} f^\pm_i,\\
        R_\text{mis}&=&\e\sum_{i\in\mathbb{Z}}\sum_{s\in\mathbb{Z}}\left[U'(s-\frac{1}{2}+v^+_{i+s}-v^-_i)(f^+_{i+s}-f^-_i)
        -U'(s-\frac{1}{2}+v^+_{i}-v^-_i)(f^+_i-f^-_i)\right]\\
        &=&\e\sum_{i\in\mathbb{Z}}\sum_{s\in\mathbb{Z}}\frac{1}{2}\left[U'(s-\frac{1}{2}+v^+_{i+s}-v^-_i)(f^+_{i+s}-f^-_i)+U'(s-\frac{1}{2}+v^+_i-v^-_{i-s})(f^+_i-f^-_{i-s})
        \right.\\
        &&\left.-2U'(s-\frac{1}{2}+v^+_{i}-v^-_i)(f^+_i-f^-_i)\right].
    \end{eqnarray*}

    2. Estimate $|R_\text{elas}|$.
    Rewrite $R_\text{elas}$ as follows
    \begin{eqnarray*}
        R_\text{elas}=-\e^{-1}\sum_{i\in\mathbb{Z}}\sum_{s\in\mathbb{Z}^*}\frac{1}{2}\left\{\e D^-_s[V'(s+\e D^+_s v_i^+)]-\e D^-_s[V'(s-\e D^+_s v_i^+)]-2\e^2 V''(s) s^2 \nabla^2 v^+_i\right\} f^+_i.
    \end{eqnarray*}
    Using Taylor expansion for $V'(\cdot)$ at $V'(s)$, we have
    \begin{eqnarray*}
        &&\e D^-_s[V'(s+\e D^+_s v_i^+)]-\e D^-_s[V'(s-\e D^+_s v_i^+)]\\
        &=&V'(s+v^+_{i+s}-v^+_i)-V'(s+v^+_i-v^+_{i-s})
        -V'(s-v^+_{i+s}+v^+_i)+V'(s-v^+_i+v^+_{i-s})\\
        &=&2(\e D^+_s v^+_i+\e D^+_{-s} v^+_i)V''(s)+\e^3[(D^+_s v^+_i)^3+(D^+_{-s} v^+_i)^3]V^{(4)}(\xi)
    \end{eqnarray*}
    for some $\xi$.
    Note that $|\e D^+_s v^+_i+\e D^+_{-s} v^+_i-\e^2s^2\nabla^2 v^+_i|\leq \frac{1}{12}\e^4 s^4 v_{4,s,i}$ and
    \begin{eqnarray*}
        \e^3|(D^+_s v^+_i)^3+(D^+_{-s} v^+_i)^3|
        \leq \e^3|D^+_s v^+_i+D^+_{-s} v^+_i|\cdot3s^2 \|\nabla v\|^2_{L^\infty}\leq 3\e^4 s^4 v_{2,s,i} \|\nabla v\|^2_{L^\infty}.
    \end{eqnarray*}
    Thus
    \begin{multline*}
        \left|\e D^-_s[V'(s+\e D^+_s v_i^+)]-\e D^-_s[V'(s-\e D^+_s v_i^+)]-2\e^2 V''(s) s^2 \nabla^2 v^+_i\right|\\
        \leq3(1+\|\nabla v\|^2_{L^\infty})(v_{2,s,i}+v_{4,s,i})\e^4 (s^4V_{2,s} +s^4V_{4,s}).
    \end{multline*}
    Therefore
    \begin{eqnarray*}
        |R_\text{elas}|
        &\leq& \e^2\frac{3}{2}(1+\|\nabla v\|^2_{L^\infty})\sum_{s\in\mathbb{Z}^*}(s^4V_{2,s} +s^4V_{4,s})\e \sum_{i\in\mathbb{Z}}(v_{2,s,i}+v_{4,s,i})|f^+_i|\\
        &\leq& C\e^2\sum_{s\in\mathbb{Z}^*}(|s|^5V_{2,s} +|s|^5V_{4,s})\|f\|_{X_\e}\\
        &\leq& C\e^2\|f\|_{X_\e}
    \end{eqnarray*}
    3. Estimate $|R_\text{mis}|$.
    Rewrite $R_\text{mis}=R_\text{mis,1}+R_\text{mis,2}$, where
    \begin{eqnarray*}
        R_\text{mis,1}&=&\e\sum_{i\in\mathbb{Z}}\sum_{s\in\mathbb{Z}}\frac{1}{2}\left[U'(s-\frac{1}{2}+v^+_{i+s}-v^-_i)
        +U'(s-\frac{1}{2}+v^+_i-v^-_{i-s})\right.\\
        &&\left.-2U'(s-\frac{1}{2}+v^+_{i}-v^-_i)\right](f^+_i-f^-_i),\\
        R_\text{mis,2}&=&\e\sum_{i\in\mathbb{Z}}\sum_{s\in\mathbb{Z}}\frac{1}{2}\left[U'(s-\frac{1}{2}+v^+_{i+s}-v^-_i)(f^+_{i+s}-f^+_i)+U'(s-\frac{1}{2}+v^+_i-v^-_{i-s})(f^-_i-f^-_{i-s})
        \right].
    \end{eqnarray*}
    Since $f\in M_\e$, we have $f^+=-f^-$ and
    \begin{eqnarray*}
        R_\text{mis,2}
        =\e\sum_{i\in\mathbb{Z}}\sum_{s\in\mathbb{Z}}\frac{1}{2}\left[U'(s-\frac{1}{2}+v^+_{i+s}-v^-_i)(f^+_{i+s}-f^+_i+f^-_{i+s}-f^-_i)\right]
        =0.
    \end{eqnarray*}
    Thanks to the symmetry of $v$, we have $U'(s-\frac{1}{2}+v^+_i-v^-_{i-s})=U'(s-\frac{1}{2}+v^+_{i-s}-v^-_i)$.
    Applying Taylor expansion, we have
    \begin{eqnarray*}
        &&\left|U'(s-\frac{1}{2}+v^+_{i+s}-v^-_i)
        +U'(s-\frac{1}{2}+v^+_{i-s}-v^-_i)-2U'(s-\frac{1}{2}+v^+_{i}-v^-_i)\right|\\
        &\leq& |v^+_{i+s}+v^+_{i-s}-2v^+_i||U''(s-\frac{1}{2}+v^+_{i}-v^-_i)|+\frac{1}{2}(|v^+_{i+s}-v^+_i|^2+|v^+_{i-s}-v^+_i|^2)U_{3,s}\\
        &\leq&\e^2 s^2 U_{2,s} v_{2,s,i}+\e^2 \|\nabla v\|_{L^\infty} s^2 U_{3,s} v_{1,s,i}.
    \end{eqnarray*}
    Thus by using Lemma \ref{lem..u^k_is.L2norm}, we obtain
    \begin{eqnarray*}
        |R_\text{mis}|=|R_\text{mis,1}|
        &\leq&\e^2(1+\|\nabla v\|_{L^\infty} )\sum_{s\in\mathbb{Z}}\left(s^2 U_{2,s}+s^2 U_{3,s}\right)\e\sum_{i\in\mathbb{Z}}(v_{2,s,i}+v_{1,s,i})|f^+_i|\\
        &\leq&C \e^2 \sum_{s\in\mathbb{Z}}\left(|s|^3 U_{2,s}+|s|^3 U_{3,s}\right) \|f\|_{X_\e}\\
        &\leq&C \e^2 \|f\|_{X_\e}.
    \end{eqnarray*}
\end{proof}

\section{Stability of the Atomistic Model}\label{sec..atom.stability}
In this section, the linear stability analysis is applied to the atomistic model. We will first study this stability at the dislocation solution of the PN model $v$, then extend it to displacement field $u$ which is sufficient close to $v$.

We start with the following key observation: with or without a dislocation, the stability gap between the atomistic and PN models remains the same, up to an $O(\e)$ truncation error.
\begin{prop}[stability gap with/without dislocation]\label{prop..stability.gap.w/o.disl}
    Suppose that Assumptions A1--A6 hold and that $\e\leq 1$. Let $v$ be the dislocation solution of the PN model in Theorem \ref{thm..cont.existence.minimizer}. For $f\in X_\e$, we have
    \begin{eqnarray}
        \left\langle \delta^2 E_\text{a}[v]f,f\right\rangle_\e-\left\langle \delta^2 E_\mathrm{PN}[v]\bar{f},\bar{f}\right\rangle
        =\left\langle \delta^2 E_\text{a}[0]f,f\right\rangle_\e-\left\langle \delta^2 E_\mathrm{PN}[0]\bar{f},\bar{f}\right\rangle
        +O(\e)\|f\|_{X_\e}^2.
    \end{eqnarray}
\end{prop}
\begin{proof}
    1. Recall second variations \eqref{eq..atom.second.variation} at continuum dislocation solution $v$
    \begin{eqnarray*}
        \left\langle \delta^2 E_\text{a}[v]f,f\right\rangle_\e
        &=&\e\sum_{i\in \mathbb{Z}}\sum_{s\in \mathbb{Z}^*}\sum_{\pm}\frac{1}{2}V''(s+\e D^+_s v^\pm_i)\left(D^+_s f^\pm_i\right)^2\nonumber\\
        &&+\e\sum_{i\in \mathbb{Z}}\sum_{s\in \mathbb{Z}}U''(s-\frac{1}{2}+v^+_{i+s}-v^-_{i})(f^+_{i+s}-f^-_i)^2,\label{eq..atom.second.variation.f.f}\\
        \left\langle \delta^2 E_\mathrm{a}[0]f,f\right\rangle_\e
        &=&\e\sum_{i\in \mathbb{Z}}\sum_{s\in \mathbb{Z}^*}\sum_\pm\frac{1}{2}V''(s)\left(D^+_s f^\pm_i\right)^2+\e\sum_{i\in \mathbb{Z}}\sum_{s\in \mathbb{Z}}U''(s-\frac{1}{2})(f^+_{i+s}-f^-_i)^2,\\
        \left\langle \delta^2 E_\mathrm{PN}[v]\bar{f},\bar{f}\right\rangle
        &=&\sum_{i\in\mathbb{Z}}\int_{\e i}^{\e(i+1)}\left\{\alpha |\nabla \bar{f}^+|^2+\alpha |\nabla \bar{f}^-|^2
        +\gamma''(v^+-v^-)(\bar{f}^\perp)^2\right\}\D x,\\
        \left\langle \delta^2 E_\mathrm{PN}[0]\bar{f},\bar{f}\right\rangle
        &=&\sum_{i\in\mathbb{Z}}\int_{\e i}^{\e(i+1)}\left\{\alpha |\nabla \bar{f}^+|^2+\alpha |\nabla \bar{f}^-|^2
        +\gamma''(0)(\bar{f}^\perp)^2\right\}\D x,
    \end{eqnarray*}
    where $\alpha=\sum_{s\in \mathbb{Z}^*} \frac{1}{2}V''(s) s^2$ and $\gamma''(\xi)=\sum_{s\in\mathbb{Z}}U''(s-\frac{1}{2}+\xi)$.
    Then
    \begin{eqnarray*}
        \left\langle \delta^2 E_\text{a}[v]f,f\right\rangle_\e-\left\langle \delta^2 E_\mathrm{a}[0]f,f\right\rangle_\e-
        \left\langle \delta^2 E_\text{PN}[v]\bar{f},\bar{f}\right\rangle+\left\langle \delta^2 E_\mathrm{PN}[0]\bar{f},\bar{f}\right\rangle
        =\sum_{k=1}^5 R_k,
    \end{eqnarray*}
    where
    \begin{eqnarray*}
        R_1&=&
        \e\sum_{i\in \mathbb{Z}}\sum_{s\in \mathbb{Z}^*}\sum_\pm\frac{1}{2}\left[
        V''(s+\e D^+_s v^\pm_i)-V''(s)\right](D^+_s f^\pm_i)^2,\\
        R_2&=&\e\sum_{i\in\mathbb{Z}}\sum_{s\in \mathbb{Z}}\left[
        U''(s-\frac{1}{2}+v^+_{i+s}-v^-_i)-U''(s-\frac{1}{2}+v^+_i-v^-_i)\right](f^+_{i+s}-f^-_i)^2,\\
        R_3&=&\e\sum_{i\in\mathbb{Z}}\sum_{s\in \mathbb{Z}}
        \left[U''(s-\frac{1}{2}+v^+_i-v^-_i)-U''(s-\frac{1}{2})\right]\left[(f^+_{i+s}-f^-_i)^2-(f^+_i-f^-_i)^2\right],\\
        R_4&=&\sum_{i\in \mathbb{Z}}\int_{\e i}^{\e (i+1)}\sum_{s\in \mathbb{Z}}\left[U''(s-\frac{1}{2}+v^+_i-v^-_i)-U''(s-\frac{1}{2}+v^+-v^-)\right](f^+_i-f^-_i)^2\D x,\\
        R_5&=&\sum_{i\in \mathbb{Z}}\int_{\e i}^{\e (i+1)}\sum_{s\in \mathbb{Z}}\left[U''(s-\frac{1}{2}+v^+-v^-)-U''(s-\frac{1}{2})\right]\left[(f^+_i-f^-_i)^2-(\bar{f}^+-\bar{f}^-)^2\right]\D x.
    \end{eqnarray*}
    Here $v^\pm=v^\pm(x)$. It remains to show $R_{i}=O(\e)\|f\|_{X_\e}^2$ for $i=1,2,\cdots,5$.\\
    2. We estimate $R_i$, $i=1,2,\cdots,5$.\\
    (1)
    Note that
    $|V''(s+\e D^+_s v^\pm_i)-V''(s)|\leq V_{3,s}|\e D^+_s v^\pm_i|\leq \e \|\nabla v\|_{L^\infty}V_{3,s} |s|$.
    Using Lemma~\ref{lem..sum.D_s.f.L2}, we have
     \begin{eqnarray*}
        |R_1|\leq \frac{1}{2}\e \|\nabla v\|_{L^\infty}\|D^+_s f\|^2_\e\sum_{s\in \mathbb{Z}^*}  V_{3,s} |s|
        \leq O(\e)\|f\|_{X_\e}^2.
    \end{eqnarray*}
    (2) Next, $(f^+_{i+s}-f^-_i)^2\leq 2(f^+_{i+s}-f^+_i)^2+2(f^\perp_i)^2\leq 2\e^2(D^+_s f^+_i)^2+2(f^\perp_i)^2
    $. Thus
    \begin{eqnarray*}
        \sum_{i\in\mathbb{Z}}(f^+_{i+s}-f^-_i)^2\leq 2\e|s|^2\|D f^+\|^2_\e+2\e^{-1}\|f^\perp\|^2_\e
        \leq \e^{-1}(2|s|^2+2)\|f\|^2_{X_\e}.
    \end{eqnarray*}
    Note that $|U''(s-\frac{1}{2}+v^+_{i+s}-v^-_i)-U''(s-\frac{1}{2}+v^+_i-v^-_i)|\leq U_{3,s}|v^+_{i+s}-v^+_i|\leq \e\|\nabla v^+\|_{L^\infty} U_{3,s} |s|$,
    Therefore
    \begin{eqnarray*}
        |R_2|&\leq&\e^2 \|\nabla v^+\|_{L^\infty}\sum_{s\in \mathbb{Z}} U_{3,s} |s| \sum_{i\in\mathbb{Z}}(f^+_{i+s}-f^-_i)^2\\
        &\leq& \e \|\nabla v^+\|_{L^\infty}\|f\|^2_{X_\e}\sum_{s\in \mathbb{Z}} U_{3,s} |s| (2|s|^2+2)\leq O(\e)\|f\|_{X_\e}^2.
    \end{eqnarray*}
    (3) Next, we have
    \begin{eqnarray}
        \sum_{i\in \mathbb{Z}}|(f^+_{i+s}-f^-_i)^2-(f^+_i-f^-_i)^2|
        &\leq&\sum_{i\in\mathbb{Z}}(f^+_{i+s}-f^+_i)^2+\sum_{i\in\mathbb{Z}}2|f^+_{i+s}-f^+_i|\cdot|f^+_i-f^-_i|\nonumber\\
        &\leq&\e^2\sum_{i\in\mathbb{Z}}|D^+_s f^+_i|^2+\e\sum_{i\in\mathbb{Z}}|D^+_s f^+_i|^2+\e\sum_{i\in\mathbb{Z}}|f^\perp_i|^2\nonumber\\
        &\leq&(\e+1)|s|^2\|D f^+\|_\e^2+\|f^\perp\|_\e^2\nonumber\\
        &\leq&2|s|^2\|f\|_{X_\e}^2,\label{eq..sum.f_{i+s}+f_i}
    \end{eqnarray}
    where we have used Lemma \ref{lem..sum.D_s.f.L2}.
    Note that $|U''(s-\frac{1}{2}+v^+_i-v^-_i)-U''(s-\frac{1}{2})|\leq \|v^\perp\|_{L^\infty} U_{3,s} $. Therefore
    \begin{eqnarray*}
        |R_3|&\leq& \e\|v^\perp\|_{L^\infty} \sum_{s\in \mathbb{Z}}U_{3,s} \sum_{i\in\mathbb{Z}}
        |(f^+_{i+s}-f^-_i)^2-(f^+_i-f^-_i)^2|\\
        &\leq&2\e\|v^\perp\|_{L^\infty}\|f\|_{X_\e}^2 \sum_{s\in \mathbb{Z}}U_{3,s} |s|^2
        \leq O(\e)\|f\|_{X_\e}^2.
    \end{eqnarray*}
    (4) We have $|U''(s-\frac{1}{2}+v^+_i-v^-_i)-U''(s-\frac{1}{2}+v^+-v^-)|\leq 2\e \|\nabla v\|_{L^\infty}U_{3,s}$. Note that $\sum_{i\in \mathbb{Z}}\int_{\e i}^{\e (i+1)} (f^+_i-f^-_i)^2\D x=\|f^\perp\|^2_\e$. Thus
    \begin{eqnarray*}
        |R_4|\leq  2\e \|\nabla v\|_{L^\infty}\|f^\perp\|^2_\e\sum_{s\in \mathbb{Z}}U_{3,s}
        \leq O(\e)\|f\|_{X_\e}^2.
    \end{eqnarray*}
    (5) Finally, we have $|U''(s-\frac{1}{2}+v^+_i-v^-_i)-U''(s-\frac{1}{2})|\leq \|v^\perp\|_{L^\infty}U_{3,s}$. Note that $|f^\perp_i-\bar{f}^\perp|=\frac{x-i\e}{\e}|f^\perp_{i+1}-f^\perp_i|=(x-i\e)|Df^+_i-Df^-_i|\leq (x-i\e)\cdot(|Df^+_i|+|Df^-_i|)$ and $|\bar{f}^\perp|\leq |f^\perp_i|+|f^\perp_{i+1}|$ for $i\e\leq x<(i+1)\e$. Hence
    \begin{eqnarray*}
        |(f^\perp_i)^2-(\bar{f}^\perp)^2|
        &\leq& |f^\perp_i-\bar{f}^\perp|\cdot(|f^\perp_i|+|\bar{f}^\perp|)\\
        &\leq& 2(x-i\e)(|Df^+_i|+|Df^-_i|)\cdot(|f^\perp_i|+|f^\perp_{i+1}|).
    \end{eqnarray*}
    Then
    \begin{eqnarray}
        \sum_{i\in \mathbb{Z}}\int_{\e i}^{\e (i+1)}|(f^\perp_i)^2-(\bar{f}^\perp)^2|\D x
        &\leq& \e^2\sum_{i\in \mathbb{Z}}(|Df^+_i|+|Df^-_i|)\cdot(|f^\perp_i|+|f^\perp_{i+1}|)\nonumber\\
        &\leq& \e(\|Df^+\|_\e+\|Df^-\|_\e)\|f^\perp\|_\e\nonumber\\
        &\leq& 2\e\|f\|_{X_\e}^2.\label{eq..f_i^2-f^2}
    \end{eqnarray}
    Therefore,
    \begin{eqnarray*}
        |R_5|\leq 2\e\|v^\perp\|_{L^\infty}\|f\|_{X_\e}^2\sum_{s\in \mathbb{Z}} U_{3,s}
        \leq O(\e)\|f\|_{X_\e}^2.
    \end{eqnarray*}
\end{proof}

Next lemma reveals the relation between a function in $X_\e$ and its extension.
\begin{lem}[linear interpolation]
    If $f\in X_\e$, then its extension $\bar{f}\in X_0$.
    Moreover, we have
    \begin{eqnarray}
        \textstyle\|D f\|_\e^2+\frac{1}{3}\|f^\perp\|_\e^2\leq \|\bar{f}\|_{X_0}^2\leq \|f\|_{X_\e}^2.\label{eq..H^1.H^1_eps}
    \end{eqnarray}
\end{lem}
\begin{proof}
    By definition, we have $\nabla \bar{f}^\pm(x)=D f^\pm_i$ for $i\e\leq x< (i+1)\e$, and hence
    $\|\nabla \bar{f}\|^2=\|D f\|_\e^2$. Direct calculation leads to
    $
        \|\bar{f}^\perp\|^2
        =\e\sum_{i\in \mathbb{Z}}\frac{1}{3}[(f^\perp_i)^2+f^\perp_if^\perp_{i+1}+(f^\perp_{i+1})^2]
    $.
    Thus $
        \frac{1}{3}\|f^\perp\|_\e^2\leq \|\bar{f}^\perp\|^2\leq\|f^\perp\|_\e^2
    $. Eq. \eqref{eq..H^1.H^1_eps} follows these immediately.
\end{proof}

\begin{prop}[explicit formula for $\Delta$]\label{prop..explicit.Delta}
    Suppose that Assumptions A1--A6 hold and that $\e \leq 1$. Let $v$ be the dislocation solution of the PN model in Theorem \ref{thm..cont.existence.minimizer}.
    For $f\in X_\e$, we have
    \begin{eqnarray}
        \left\langle \delta^2 E_\text{a}[0]f,f \right\rangle_{\e}-\left\langle \delta^2 E_\text{PN}[0]\bar{f},\bar{f} \right\rangle
        \geq -\Delta \|Df\|_\e^2+O(\e)\|f\|_{X_\e}^2.
    \end{eqnarray}
    Moreover, $\Delta$ can be calculated by
    \begin{eqnarray}
        \Delta
        =\sup_{\|Df\|_\e=1}\left\{\e \sum_{i\in \mathbb{Z}}\sum_{s\geq 2}\sum_\pm V''(s)\left[\left(D^+_s f^\pm_i\right)^2-s^2(D f^\pm_i)^2\right]\right\}.\label{eq..Delta.explicit}
    \end{eqnarray}
\end{prop}
\begin{proof}
    By direct calculations, we have
    \begin{eqnarray*}
        \left\langle \delta^2 E_\text{a}[0]f,f \right\rangle_{\e}-\left\langle \delta^2 E_\text{PN}[0]\bar{f},\bar{f} \right\rangle
        &=& \e\sum_{i\in \mathbb{Z}}\left[\sum_\pm\sum_{s\in \mathbb{Z}^*}\frac{1}{2}V''(s)\left(D^+_s f^\pm_i\right)^2
        +\sum_{s\in \mathbb{Z}}U''(s-\frac{1}{2})(f^+_{i+s}-f^-_i)^2\right]\\
        &&-\sum_{i\in\mathbb{Z}}\int_{\e i}^{\e(i+1)}\left[\sum_\pm\sum_{s\in \mathbb{Z}^*}\frac{1}{2}V''(s)s^2|\nabla \bar{f}^\pm|^2-\sum_{s\in \mathbb{Z}}U''(s-\frac{1}{2})(\bar{f}^\perp)^2\right]\D x.
    \end{eqnarray*}
    Let
    \begin{eqnarray*}
        \tilde{R}_1&=&\e\sum_{i\in\mathbb{Z}}\sum_{s\in \mathbb{Z}}
        U''(s-\frac{1}{2})\left[(f^+_{i+s}-f^-_i)^2-(f^+_i-f^-_i)^2\right],\\
        \tilde{R}_2&=&\sum_{i\in \mathbb{Z}}\int_{\e i}^{\e (i+1)}\sum_{s\in \mathbb{Z}}U''(s-\frac{1}{2})\left[(f^\perp_i)^2-(\bar{f}^\perp)^2\right]\D x.
    \end{eqnarray*}
    Recall Eqs. \eqref{eq..sum.f_{i+s}+f_i} and \eqref{eq..f_i^2-f^2}, thus we have
    \begin{eqnarray*}
        |\tilde{R}_1|&\leq& 2\e\|f\|_{X_\e}^2\sum_{s\in \mathbb{Z}}U_{2,s}|s|^2\leq O(\e)\|f\|_{X_\e}^2,\\
        |\tilde{R}_2|&\leq& 2\e\|f\|_{X_\e}^2\sum_{s\in \mathbb{Z}}U_{2,s} \leq O(\e)\|f\|_{X_\e}^2.
    \end{eqnarray*}
    Note that $\nabla \bar{f}^\pm(x)=D f^\pm_i$ for $i\e\leq x< (i+1)\e$. Recall the definition \eqref{eq..Delta.definition}. Therefore,
    \begin{eqnarray*}
        \Delta&=&\lim_{\e\rightarrow 0}\sup_{\|Df\|_\e=1}\left\langle \delta^2 E_\text{PN}[0]\bar{f},\bar{f} \right\rangle-\left\langle \delta^2 E_\text{a}[0]f,f \right\rangle_{\e}\\
        &=& \lim_{\e\rightarrow 0}\sup_{\|Df\|_\e=1}\left\{\e\sum_{i\in \mathbb{Z}}\sum_{s\in \mathbb{Z}^*}\sum_\pm\frac{1}{2}V''(s)\left[\left(D^+_s f^\pm_i\right)^2-s^2(D f^\pm_i)^2\right]-\tilde{R}_1-\tilde{R}_2\right\}\\
        &=& \sup_{\|Df\|_\e=1}\e\sum_{i\in \mathbb{Z}}\sum_{s\geq 2}\sum_\pm V''(s)\left[\left(D^+_s f^\pm_i\right)^2-s^2(D f^\pm_i)^2\right],
    \end{eqnarray*}
    where we use the symmetry of $V$ (Assumption A2) in the last step.
\end{proof}
\begin{prop}[$\Delta\geq 0$]\label{prop..Delta.geq.0}
    The stability gap \eqref{eq..Delta.explicit} is non-negative: $\Delta\geq 0$.
\end{prop}
\begin{proof}
    By Lemma \ref{lem..decay.properties.V.U}, we have
    $\sum_{s\geq 2}|V''(s)|s^2\leq \sum_{s\in \mathbb{Z}^*} V_{2,s}s^2<C$. Then for any $M\in \mathbb{N}^*$,
    there exists an $t\in \mathbb{N}^*$ such that $\sum_{s\geq t+1}|V''(s)|s^2<\frac{1}{M}$. For $s\geq 2$,
    by the Cauchy--Schwarz inequality, we obtain
    \begin{eqnarray}
        \sum_{i\in \mathbb{Z}}(D^+_s f^\pm_i)^2
        \leq \sum_{i\in \mathbb{Z}}s\sum_{j=i}^{i+s-1}(D f^\pm_j)^2
        =s^2\sum_{i\in \mathbb{Z}}(D f^\pm_i)^2.\label{eq..D^s.D.Cauchy.Schwatz}
    \end{eqnarray}
    We define $g$ as follows: $g_i=(2\e Mt)^{-1/2}$ for $1\leq i\leq Mt$ and $g_i$ otherwise. Obviously, $\|g\|_\e^2=\frac{1}{2}$.
    Note that if we define $Df^\pm=\pm g$, then $\|Df\|_\e=1$. Therefore
    \begin{eqnarray*}
        \Delta\geq \e \sum_{i\in \mathbb{Z}}\sum_{s\geq 2}2 V''(s)\left[(g_i+\cdots+g_{i+s-1})^2-s^2 g_i ^2\right].
    \end{eqnarray*}
    If $2\leq s\leq t$, then $(g_i+\cdots+g_{i+s-1})^2-s^2g_i^2=0$ for $i\not\in T$, where $T=\{-s+2,-s+2,\cdots,0\}\cup\{Mt-s+2,Mt-s+3,\cdots,Mt\}$.
    For $i\in T$, we have $|(g_i+\cdots+g_{i+s-1})^2-s^2g_i^2|\leq s^2 (2\e Mt)^{-1}$. Note that $|T|=2(s-1)$.
    Thus for any $2\leq s\leq t$, we have
    $
        \e \sum_{i\in \mathbb{Z}}\left[(g_i+\cdots+g_{i+s-1})^2-s^2 g_i ^2\right]
        \geq-\e 2(s-1)s^2(2\e Mt)^{-1}\geq -\frac{s^2}{M}
    $.

    If $s\geq t+1$, Eq. \eqref{eq..D^s.D.Cauchy.Schwatz} implies that
    $\e\sum_{i\in \mathbb{Z}}[(g_i+\cdots+g_{i+s-1})^2-s^2 g_i ^2]\geq -\e\sum_{i\in \mathbb{Z}} s^2 g_i^2=- \frac{s^2}{2}$.

    Therefore,
    \begin{eqnarray*}
        \Delta
        &\geq&\e \sum_{i\in \mathbb{Z}}\left\{\sum_{s=2}^{t}+\sum_{s=t+1}^\infty\right\}
        2 V''(s)\left[(g_i+\cdots+g_{i+s-1})^2-s^2 g_i ^2\right]\\
        &\geq& -\sum_{s=2}^t 2|V''(s)|\frac{s^2}{M}-\sum_{s=t+1}^{\infty}2|V''(s)|
        \frac{s^2}{2}\\
        &\geq& -\frac{1+2C}{M}.
    \end{eqnarray*}
    Let $M$ go to infinity, we obtain $\Delta \geq 0$.
\end{proof}
\begin{prop}\label{prop..NN.Delta}
    Suppose that Assumptions A1--A6 hold.
    If $V''(s)\leq 0$ for all $|s|\geq 2$, then
    $
        \Delta = 0
    $, and thence $\kappa>\Delta$. In particular, if $V(\cdot)$ is a nearest neighbor potential then $\kappa>\Delta=0$.
\end{prop}
\begin{proof}
Eq. \eqref{eq..D^s.D.Cauchy.Schwatz} and $V''(s)\leq 0$ imply that
$
    V''(s)\sum_{i\in\mathbb{Z}}\left[\left(D^+_s f^\pm_i\right)^2-s^2(D f^\pm_i)^2\right]\leq 0
$ for $|s|\geq 2$.
Hence $\Delta\leq 0$.
According to Proposition \ref{prop..Delta.geq.0}, we have
$
    \Delta=0
$.
\end{proof}

\begin{prop}[stability of atomistic model]\label{prop..atom.stability}
    Suppose that Assumptions A1--A7 hold. Let $v$ be the dislocation solution of the PN model in Theorem \ref{thm..cont.existence.minimizer}.
    There exist $C$ and $\e_0$ such that for $0<\e<\e_0$ and for all $f\in X_\e$, we have
    \begin{eqnarray}\label{eq..atomistic.stability}
        \textstyle \left\langle \delta^2 E_\text{a}[v]f, f\right\rangle_\e\geq C\|f\|_{X_{\e}}^2.
    \end{eqnarray}
    Here $C$ and $\e_0$ depend on $R$, $\theta$, $\alpha$, $\gamma''(0)$, and $\Delta$.
\end{prop}
\begin{proof}
    By Propositions \ref{prop..stability.gap.w/o.disl} and \ref{prop..explicit.Delta}, we have
    \begin{eqnarray*}
        \left\langle \delta^2 E_\text{a}[v]f,f\right\rangle_\e
        &=&\left\langle \delta^2 E_\text{PN}[v]\bar{f},\bar{f}\right\rangle+\left\langle \delta^2 E_\text{a}[0]f,f\right\rangle_\e-\left\langle \delta^2 E_\text{PN}[0]\bar{f},\bar{f}\right\rangle+O(\e)\|f\|_{X_\e}^2\\
        &\geq&\textstyle\kappa \|Df\|_\e^2+\frac{1}{3}\kappa\|f^\perp\|_\e^2-\Delta \|Df\|_\e^2+O(\e)\|f\|_{X_\e}^2\\
        &\geq& C\|f\|_{X_\e}^2
    \end{eqnarray*}
    for sufficiently small $\e$. Here we utilize the assumption $\kappa>\Delta$.
\end{proof}

We finish this section with a detailed verification on the stability condition of Lennard--Jones $(m,n)$ potential. The commonly used case is $(m,n)=(6,12)$.

\begin{prop}\label{prop..LJ.Delta}
    Let $V(\cdot)$ be Lennard--Jones $(m,n)$ potential, i.e.,
    \begin{eqnarray}
        V(x)=V_\text{LJ}(x)=-\left(\frac{r_0}{|x|}\right)^{m}+\left(\frac{r_0}{|x|}\right)^{n},\,\,1<m<n,\,\,x\neq 0,
    \end{eqnarray}
    where $r_0$ is some characteristic distance. Then $\Delta=0$, provided $\e$ is sufficiently small.
\end{prop}
\begin{proof}
We first remark that $r_0$ is not arbitrary but related to the minimal distance $s_0=1$ (the rescaled lattice constant). Note that $s_0=1$ solves
\begin{eqnarray}
    \frac{\partial}{\partial s_0}\left(\sum_{k\in \mathbb{Z}^*}V(ks_0)+\sum_{k\in \mathbb{Z}}V_d(ks_0-\frac{1}{2}s_0)\right)=0
\end{eqnarray}
Recall that $V_d=\e^2 U$. Thus
\begin{eqnarray}\label{eq..solving.s_0}
    \sum_{k\in \mathbb{Z}^*}kV'(k)+\e^2 \sum_{k\in \mathbb{Z}}(k-\frac{1}{2})U'(k-\frac{1}{2})=0.
\end{eqnarray}
By Lemma \ref{lem..decay.properties.V.U}, we have
$
    |\sum_{k\in \mathbb{Z}}(k-\frac{1}{2})U'(k-\frac{1}{2})|
    \leq\sum_{s\in \mathbb{Z}}(|s|+1)U_{1,s}\leq C
$.
Then
\begin{eqnarray*}
    0=\sum_{k\in \mathbb{Z}^*}kV'(k)+O(\e^2)
    &=&\sum_{k\in \mathbb{Z}^*}\left[m\frac{r_0^m}{k^m}-n\frac{r_0^{n}}{k^{n}}\right]+O(\e^2)\\
    &=&2m\zeta(m)r_0^m-2n\zeta(n)r_0^{n}+O(\e^2),
\end{eqnarray*}
where the zeta function $\zeta(t)=\sum_{k=1}^{\infty}k^{-t}$, $t>1$.
Therefore, for sufficient small $\e$, we have
\begin{eqnarray*}
    r_0^{n-m}=\frac{m\zeta(m)}{n\zeta(n)}+O(\e^2).
\end{eqnarray*}

For $s\geq 2$, we have
\begin{eqnarray*}
    V''(s)
    &=&m(m+1)\frac{r_0^m}{s^{m+2}}\left[-1+\frac{n(n+1)r_0^{n-m}}{m(m+1)s^{n-m}}\right]\\
    &\leq &m(m+1)\frac{r_0^m}{s^{m+2}}\left[-1+\frac{n(n+1)}{m(m+1)}\cdot\frac{\frac{m\zeta(m)}{n\zeta(n)}+O(\e^2)}{2^{n-m}}\right].
\end{eqnarray*}
It can be shown that $\frac{(n+1)\zeta(m)}{(m+1)\zeta(n)}<2^{n-m}$. Hence $V''(s)\leq 0$, $s\geq 2$ for sufficiently small $\e$.
By Propositions \ref{prop..Delta.geq.0} and \ref{prop..NN.Delta}, we obtain $\Delta=0$.
\end{proof}

\section{Existence of the Atomistic Model and Convergence}
In this section, we show that the atomistic model has a solution $v^\e$ which is $O(\e^2)$ away from the PN solution $v$ in terms of the metric induced by $X_{\e}$ norm.
Let us first provide the following lemma which makes use of the continuity of $\delta^2 E_\text{a}$ at $v$.

\begin{lem}\label{lem..nonexpansive}
    Suppose that Assumptions A1--A6 hold. Let $v$ be the dislocation solution of the PN model in Theorem \ref{thm..cont.existence.minimizer}. There exist constants $\e_0$ and $C$ such that for any $0<\e<\e_0$ and any $u, u'$ satisfying $u-v\in X_\e$, $u'-v\in X_\e$, $\|u-v\|_{X_\e}\leq \e$ and $\|u'-v\|_{X_\e}\leq \e$, we have
    \begin{eqnarray}
        |\left\langle \left(\delta^2 E_\text{a}[u]-\delta^2 E_\text{a}[u']\right)f, g\right\rangle_\e|\leq C\e^{-1/2}\|u-u'\|_{X_\e}\|f\|_{X_\e}\|g\|_{X_\e}\label{eq..Hu.Hu'.difference.f.g}
    \end{eqnarray}
    for all $f, g \in X_\e$. Moreover, if $u-v\in M_\e$, $u'-v\in M_\e$, and $f\in M_\e$, then we have
    \begin{eqnarray}
        \|\left(\delta^2 E_\text{a}[u]-\delta^2 E_\text{a}[u']\right)f\|_\e
        \leq C\e^{-3/2}\|u-u'\|_{X_\e}\|f\|_{X_\e}.\label{eq..Hu.Hu'.difference}
    \end{eqnarray}
    Here $c$ and $C$ depend on $R$, $\theta$, $\alpha$, $\gamma''(0)$, and $\Delta$.
\end{lem}
\begin{proof}
    Setting $g=\left(\delta^2 E_\text{a}[u]-\delta^2 E_\text{a}[u']\right)f$ in Eq. \eqref{eq..Hu.Hu'.difference.f.g}, we have
    \begin{eqnarray*}
        \|g\|_\e^2
        \leq C \|u-u'\|_{X_\e}\|f\|_{X_\e}\|g\|_{X_\e}
    \end{eqnarray*}
    Utilizing Lemmata \ref{lem..space.L2.epsilon} and \ref{lem..subspace.M}, we obtain $\|g\|_{X_\e}\leq C\|g\|_{\e,1} \leq C\e^{-1} \|g\|_\e$. Thus
    \begin{eqnarray*}
        \|g\|_\e^2
        &\leq& C\e^{-3/2}\|u-u'\|_{X_\e}\|f\|_{X_\e}\|g\|_\e.
    \end{eqnarray*}
    This leads to Eq. \eqref{eq..Hu.Hu'.difference}. It remains to show Eq. \eqref{eq..Hu.Hu'.difference.f.g}.

    Note that $\|D^+_s(u-v)\|_{L^\infty_\e}\leq|s|\|D(u-v)\|_{L^\infty_\e}\leq |s|\e^{-1/2}\|u-v\|_{X_\e}\leq|s|\e^{1/2}$. This with $\|D^+_s v\|_{L^\infty_\e}\leq |s|\|v\|_{C^{0,1}}\leq C|s|$ implies $\|D^+_s u\|_{L^\infty_\e}\leq C|s|$.
    Similarly, we have $\|D^+_s(u'-v)\|_{L^\infty_\e}\leq |s|\e^{-1/2}\|u'-v\|_{X_\e} \leq |s|\e^{1/2}$, $\|D^+_s(u'-u)\|_{L^\infty_\e}\leq |s|\e^{-1/2}\|u'-u\|_{X_\e}\leq |s|\e^{1/2}$, and $\|D^+_s u'\|_{L^\infty_\e}\leq C|s|$.
    For sufficiently small $\e$, we have
    \begin{eqnarray*}
        |V''(s+\e D^+_s u^\pm_i)-V''(s+\e D^+_s u'^\pm_i)|
        =|V^{(3)}(\xi)||\e D^+_s(u'^\pm_i-u^\pm_i)|
        \leq V_{3,s}|s|\e^{1/2}\|u'-u\|_{X_\e},
    \end{eqnarray*}
    where $|\xi-s|\leq \max\{|\e D^+_s u^\pm_i|,|\e D^+_s u'^\pm_i|\}\leq C\e^{1/2}|s|\leq \frac{1}{2}|s|$.

    Note that $\|u^\perp-v^\perp\|_{L^\infty_\e}\leq \e^{-1/2}\|u-v\|_{X_\e}\leq \e^{1/2}$. This with $\|v^\perp\|_{L^\infty_\e}\leq 1$ implies that $\|u^\perp\|_{L^\infty_\e}\leq 1+\e^{1/2}\leq 2$. Similarly, we have $\|u'^\perp-v^\perp\|_{L^\infty_\e}\leq \e^{-1/2}\|u'-v\|_{X_\e}\leq \e^{1/2}$, $\|u'^\perp-u^\perp\|_{L^\infty_\e}\leq 2\e^{-1/2}\|u'-u\|_{X_\e}\leq 2\e^{1/2}$, and $\|u'^\perp\|_{L^\infty_\e}\leq 2$. For sufficiently small $\e$, we have
    \begin{eqnarray*}
        &&|U''(s-\frac{1}{2}+u^+_{i+s}-u^-_i)-U''(s-\frac{1}{2}+u'^+_{i+s}-u'^-_i)|\\
        &\leq& |U^{(3)}(\xi)||\e D^+_s(u'^+_i-u^+_i)+(u'^\perp_i-u^\perp_i)|\\
        &\leq &\left(\sum\nolimits_{j=-|s|-2}^{|s|+2} U_{3,s+j}\right)(|s|+2)\e^{-1/2}\|u'-u\|_{X_\e},
    \end{eqnarray*}
    where we have used that
    $
        |\xi-(s-\frac{1}{2})|
        \leq \max\{|\e D^+_s u'^+_i|+|u'^\perp_i|,|\e D^+_s u^+_i|+|u^\perp_i|\}
        \leq |s|+2
    $ and that $\sup_{|\xi-(s-\frac{1}{2})|\leq |s|+2}|U^{(3)}(\xi)|\leq \sum_{j=-|s|-2}^{|s|+2} U_{3,s+j}$.

    Recall Eq. \eqref{eq..atom.second.variation} and hence we have
    \begin{eqnarray*}
        &&|\left\langle \left(\delta^2 E_\text{a}[u]-\delta^2 E_\text{a}[u']\right)f, g\right\rangle_\e|\\
        &\leq& \e^{-1/2}\|u-u'\|_{X_\e} \cdot\frac{\e}{2}\sum_{\pm}\sum_{i\in \mathbb{Z}}\sum_{s\in \mathbb{Z}^*} V_{3,s}|s| \left|D^+_s f^\pm_i\right| \cdot \left|D^+_s g^\pm_i\right| \\
        &&+\e^{-1/2}\|u-u'\|_{X_\e}\cdot \e\sum_{i\in \mathbb{Z}}\sum_{s\in \mathbb{Z}}\left(\sum\nolimits_{j=-|s|-2}^{|s|+2} U_{3,s+j}\right)(|s|+2) \left|f^+_{i+s}-f^-_i\right|\cdot\left|g^+_{i+s}-g^-_i\right|.
    \end{eqnarray*}
    Utilizing Lemmas \ref{lem..decay.properties.V.U} and \ref{lem..sum.D_s.f.L2}, we obtain
    \begin{eqnarray*}
        &&\frac{\e}{2}\sum_{s\in \mathbb{Z}^*}V_{3,s}|s| \sum_{i\in \mathbb{Z}}\left|D^+_s f^\pm_i\right|\cdot\left|D^+_s g^\pm_i\right|
        \leq \frac{1}{2}\sum_{s\in \mathbb{Z}^*}V_{3,s}|s|^3\|D f^\pm\|_\e \|D g^\pm\|_\e
        \leq C \|f\|_{X_\e}\|g\|_{X_\e},\\
        &&\e\sum_{s\in \mathbb{Z}}\left(\sum\nolimits_{j=-|s|-2}^{|s|+2} U_{3,s+j}\right)(|s|+2) \sum_{i\in \mathbb{Z}}\left|f^+_{i+s}+f^-_i\right|\cdot\left|g^+_{i+s}+g^-_i\right|
        \leq C \|f\|_{X_\e}\|g\|_{X_\e}.
    \end{eqnarray*}
    Finally, Eq. \eqref{eq..Hu.Hu'.difference.f.g} is obtained by collecting these inequalities.
\end{proof}
\begin{lem}\label{lem..atom.stability.u}
    Suppose that Assumptions A1--A7 hold. Let $v$ be the dislocation solution of the PN model in Theorem \ref{thm..cont.existence.minimizer}. There exist constants $\e_0$ and $C$ such that for any $0<\e<\e_0$ and any $u$ satisfying $u-v\in X_\e$ and $\|u-v\|_{X_{\e}}\leq \e$, we have
    \begin{eqnarray}
        \textstyle \left\langle \delta^2 E_\text{a}[u]f, f\right\rangle_\e\geq C\|f\|_{X_{\e}}^2\label{eq..atomistic.stability.u}
    \end{eqnarray}
    for all $f\in X_\e$.
    Here $c$ and $C$ depend on $R$, $\theta$, $\gamma''(0)$, $\Delta$, and $\kappa$.
\end{lem}
\begin{proof}
    Thanks to Proposition \ref{prop..atom.stability}, we know $\left\langle \delta^2 E_\text{a}[v]f, f\right\rangle_\e\geq C\|f\|_{X_\e}^2$ for all $f\in X_{\e}$.
    It is sufficient to show that $|\left\langle \delta^2 E_\text{a}[v]f, f\right\rangle_\e-\left\langle \delta^2 E_\text{a}[u]f, f\right\rangle_\e|\leq \frac{1}{2}C\|f\|_{X_\e}^2$. The latter can be obtained by setting $v=u'$ in Lemma \ref{lem..nonexpansive}.
\end{proof}

As all preparations are complete, we provide a proof of our main theorem.

\begin{proof}[Proof of Theorem 2]
    By Theorem \ref{thm..cont.existence.minimizer}, we have $v\in C^5$ and $\|\nabla v\|_{W^{4,1}}\leq C$ independent of $\e$.
    Define a closed ball $B$ of $M_\e$:
    \begin{eqnarray}
        B=\left\{w\in M_\e:\|w\|_{X_\e}\leq C_0\e^2\right\},
    \end{eqnarray}
    where the constant $C_0$ can be chosen properly later. Given $w\in B$, we define operator $A_{w}: M_\e \rightarrow M_\e$ as follows
    \begin{eqnarray}
        \langle A_{w}f,g\rangle_{\e}=\int_{0}^{1}\langle \delta^2 E_\text{a}[u^t]f,g\rangle_\e\D t, \,\,f,g\in M_\e,
    \end{eqnarray}
    where $u^t=v+tw$ for $t\in [0,1]$. It is easy to check that this operator is well-defined.
    Next, we have $\|u^t-v\|_{X_\e}=t\|w\|_{X_\e}\leq C_0\e^2$. Then by Lemma \ref{lem..atom.stability.u}, we have $\left\langle \delta^2 E_\text{a}[u^t]f,f\right\rangle_{\e}\geq C\|f\|_{X_\e}^2$ for $t\in [0,1]$ and $f\in M_\e\subset X_\e$. Thus
    $
        \langle A_w f, f\rangle_{\e} \geq C\|f\|^2_{X_\e}
        \geq C\|f\|^2_\e
    $
    and $A_{w}$ is invertible. Since $-\delta E_\text{a}[v]\in M_\e$, we have $-A_w^{-1}\delta E_\text{a}[v]\in M_\e$.

    By Taylor's theorem with remainder, we have
    \begin{eqnarray}
        \delta E_\text{a}[v+w]
        =\delta E_\text{a}[v]+\int_{0}^{1}\delta^2 E_\text{a}[u^t]w\D t
        =\delta E_\text{a}[v]+A_w w,
    \end{eqnarray}
    where $w\in B$ and $u^t=v+tw$ for $t\in [0,1]$.

    To solve the atomistic model, we are sufficient to find $w\in B$ solving $A_{w}w=-\delta E_\text{a}[v]$. Define a map $G: B\rightarrow M_\e$
    for any $w\in B$,
    \begin{eqnarray}
        G(w)=-A_w^{-1}\delta E_\text{a}[v].
    \end{eqnarray}

    Next, we check that $G(B)\subset B$. Indeed, by Lemma \ref{lem..atom.stability.u} and the consistency (Proposition \ref{prop..consistency}), we have
    \begin{eqnarray*}
        C\|G(w)\|^2_{X_{\e}}
        &\leq&\langle A_w G(w),G(w)\rangle_{\e}\\
        &\leq&|\langle\delta E_\text{a}[v],G(w)\rangle_{\e}|\\
        &\leq&O(\e^2)\|G(w)\|_{X_{\e}},\\
        \|G(w)\|_{X_{\e}}&\leq& C_0\e^2.
    \end{eqnarray*}

    We are going to apply the contraction mapping theorem to $G$. Obviously, $B$ is a non-empty complete metric space with metric $d(u,v)=\|u-v\|_{X_\e}$. To guarantee the existence (and uniqueness) of a fixed point in $B$, it remains to show that $G: B\rightarrow B$ is a contraction mapping, i.e., $\|G(w)-G(w')\|_{X_\e}\leq L\|w-w'\|_{X_\e}$ for any $w, w'\in B$ and for some Lipschitz constant $L<1$.
    \begin{eqnarray*}
        \|G(w)-G(w')\|_{X_\e}
        &=&\|\left(A_w^{-1}-A_{w'}^{-1}\right)\delta E_\text{a}[v]\|_{X_{\e}}\\
        &=&\|A_w^{-1}(A_w-A_{w'})A_{w'}^{-1}\delta E_\text{a}[v]\|_{X_\e}\\
        &\leq& P_1
        \cdot P_2
        \cdot P_3
        \cdot\|\delta E_\text{a}[v]\|_\e,
    \end{eqnarray*}
    where
    \begin{eqnarray*}
        P_1&:=&\frac{\|A_w^{-1}(A_w-A_{w'})A_{w'}^{-1}\delta E_\text{a}[v]\|_{X_\e}}{\|(A_w-A_{w'})A_{w'}^{-1}\delta E_\text{a}[v]\|_\e}
        \leq \sup_{f\in M_\e}\frac{\|f\|_{X_\e}}{\|A_w f\|_\e},\\
        P_2&:=&\frac{\|(A_w-A_{w'})A_{w'}^{-1}\delta E_\text{a}[v]\|_\e}{\|A_{w'}^{-1}\delta E_\text{a}[v]\|_{X_\e}}
        \leq \sup_{f\in M_\e}\frac{\|(A_w-A_{w'})f\|_\e}{\|f\|_{X_\e}},\\
        P_3&:=& \frac{\|A_{w'}^{-1}\delta E_\text{a}[v]\|_{X_\e}}{\|\delta E_\text{a}[v]\|_\e}
        \leq \sup_{f\in M_\e}\frac{\|f\|_{X_\e}}{\|A_{w'} f\|_\e}.
    \end{eqnarray*}

    For $f\in M_\e$, $f\neq 0$ and $w\in B$, we have
    \begin{eqnarray*}
        C \|f\|_{X_\e}
        \leq \frac{\langle A_w f, f\rangle_\e}{\|f\|_{X_\e}}
        \leq \frac{\langle A_w f, f\rangle_\e}{\|f\|_\e}
        \leq \|A_w f\|_\e.
    \end{eqnarray*}
    Hence
    \begin{eqnarray*}
        P_1\leq C,\,\,P_3\leq C.
    \end{eqnarray*}

    By Lemma \ref{lem..nonexpansive}, we have
    \begin{eqnarray*}
        \|(A_w-A_{w'})f\|_\e
        &\leq&\int_{0}^{1}\|(\delta^2 E_\text{a}[v+tw]-\delta^2 E_\text{a}[v+tw'])f\|_\e\D t\\
        &\leq&\int_{0}^{1}C\e^{-3/2}\|tw-tw'\|_{X_\e}\|f\|_{X_\e}\D t\\
        &\leq&C\e^{-3/2}\|w-w'\|_{X_\e}\|f\|_{X_\e}.
    \end{eqnarray*}
    Thus
    \begin{eqnarray*}
        P_2\leq C\e^{-3/2}\|w-w'\|_{X_\e}.
    \end{eqnarray*}

    Combining these estimates with $\|\delta E_\text{a}[v]\|_\e\leq C\e^2$, we obtain
    \begin{eqnarray}
        \|G(w)-G(w')\|_{X_\e}
        \leq C\e^{-3/2}\|w-w'\|_{X_\e} C\e^2
        \leq L \|w-w'\|_{X_\e},
    \end{eqnarray}
    where $L<1$ for sufficiently small $\e$. Therefore, $G$ is a contraction mapping. There exists a unique fixed point $w^\e$ solving $H_{w^\e}w^\e=-\delta E_\text{a}[v]$. Let $v^\e=v+w^\e$.
    Then $v^\e$ solves the Euler--Lagrange equation of the atomistic model and satisfies $\|v^{\e}-v\|_{X_{\e}}\leq C\e^2$.
    Finally, $v^\e$ is a local minimizer of $E_\text{a}$ in $X_\e$ norm.
    In fact, for any $w\in X_\e$ with $\|w\|_{X_\e}\leq C_0\e^2$, we apply Lemma \ref{lem..atom.stability.u} and obtain
    \begin{eqnarray*}
        E_\text{a}[v^\e+w]-E_\text{a}[v^\e]=\int_0^1(1-t)\langle\delta^2 E_\text{a}[v^\e+tw]w,w\rangle_\e \D t\geq C\|w\|_{X_\e}^2>0.
    \end{eqnarray*}
\end{proof}

\begin{proof}[Proof of Corollary 1]
    1. We suppose, without loss of generality, that $\e\leq 1$. Since $v^+=-v^-$, the total energy of the PN model at $v$ reads as
    \begin{eqnarray}
        E_\text{PN}[v]=\int_{\mathbb{R}}\left[\alpha |\nabla v^+|^2+\gamma(2v^+)\right]\D x.
    \end{eqnarray}
    Using trapezoidal rule, we have the numerical approximation of this energy
    \begin{eqnarray}
        E^\text{app}_\text{PN}[v]=\e\sum_{i\in\mathbb{Z}}\left[\alpha |\nabla v^+_i|^2+\gamma(2v^+_i)\right].
    \end{eqnarray}
    It is sufficient to show that $\left|E_\text{a}[v^\e]-E^\text{app}_\text{PN}[v]\right|\leq C\e^2$ and $\left|E^\text{app}_\text{PN}[v]-E_\text{PN}[v]\right|\leq C\e^2$.

    2. Estimate $|E_\text{a}[v^\e]-E^\text{app}_\text{PN}[v]|$. Recall Eqs. \eqref{eq..alpha.rescaled} and \eqref{eq..gamma.rescaled}. Let $E_\text{a}[v^\e]-E^\text{app}_\text{PN}[v]=R_\text{elas}+R_\text{mis}$, where
    \begin{eqnarray*}
        R_\text{elas}&=&\frac{\e^{-1}}{2}\sum_{i\in\mathbb{Z}}\sum_{s\in\mathbb{Z}^*}
        \left[V(s+\e D^+_s v^{\e,+}_i)+V(s-\e D^+_s v^{\e,+}_i)-2V(s)-\e^2V''(s)s^2(\nabla v^+_i)^2\right],\\
        R_\text{mis}&=&\e\sum_{i\in\mathbb{Z}}\sum_{s\in\mathbb{Z}}\left[U(s-\frac{1}{2}+v^{\e,+}_{i+s}+v^{\e,+}_i)-U(s-\frac{1}{2}+2v^+_i)\right].
    \end{eqnarray*}
    Let $w=v^\e-v$ on $\e\mathbb{Z}$. Thanks to Theorem \ref{thm..atom.existence.minimizer}, we have $w\in M_\e$ and $\|w\|_{X_\e}\leq C\e^2$. This implies that $v^{\e,+}=-v^{\e,-}$, $\|Dw\|_{L^\infty_\e}\leq C\e^\frac{3}{2}$, and $\|Dw\|_{\e}\leq C\e^2$.
    Using Lemmas \ref{lem..sum.D_s.f.L2} and \ref{lem..u^k_is.L2norm}, we have $\|D^+_s w\|_{\e}\leq |s|\|D w\|_{\e}\leq C|s|\e^2$ and $\|D^+_s v\|_{\e}\leq |s| \|D v\|_{\e}\leq |s|\|v_{1,1}\|_\e\leq C|s|$. Also notice that $\|D^+_s v\|_{L^\infty_\e}\leq |s|\|v\|_{C^{0,1}}\leq C|s|$ and $\|D^+_s w\|_{L^\infty_\e}\leq |s|\|D w\|_{L^\infty_\e}\leq C|s|\e^{\frac{3}{2}}$.
    Thus
    \begin{eqnarray*}
        &&\|D^+_s v^{\e}\|_{\e}\leq \|D^+_s v\|_{\e}+\|D^+_s w\|_{\e} \leq C|s|,\\
        &&\|D^+_s v^{\e}\|_{L^\infty_\e}\leq \|D^+_s v\|_{L^\infty_\e}+\|D^+_s w\|_{L^\infty_\e} \leq C|s|.
    \end{eqnarray*}
    Since $\|D^+_s w\|_\e\leq C|s|\e^2$, we have $\|D^-_s D^+_s w\|_\e\leq |s|\|D D^+_s w\|_\e\leq C\e^{-1}|s|\|D^+_s w\|_\e\leq C s^2 \e$. Note that $\|D^-_s D^+_s v\|_\e\leq s^2\|v_{2,1}\|_\e\leq Cs^2$. Thus
    \begin{eqnarray*}
        \|D^-_s D^+_s v^\e\|_\e\leq \|D^-_s D^+_s w\|_\e+\|D^-_s D^+_s v\|_\e\leq Cs^2.
    \end{eqnarray*}

    To estimate the elastic part, we apply Taylor theorem:
    \begin{eqnarray}
        |R_\text{elas}|\leq \left|\frac{\e}{2}\sum_{s\in\mathbb{Z}^*}V''(s)\sum_{i\in\mathbb{Z}}\left[(D^+_s v^{\e,+}_i)^2-(s\nabla v^+_i)^2\right]\right|+\frac{\e^{3}}{24}\sum_{s\in\mathbb{Z}^*}V_{4,s}\sum_{i\in\mathbb{Z}}|D^+_s v^{\e,+}_i|^4.\label{eq..energy.consistency1}
    \end{eqnarray}
    For the second term on the right hand side of \eqref{eq..energy.consistency1}, we have
    \begin{eqnarray}
        \frac{\e^{3}}{24}\sum_{s\in\mathbb{Z}^*}V_{4,s}\sum_{i\in\mathbb{Z}}|D^+_s v^{\e,+}_i|^4
        \leq C\e^2\sum_{s\in\mathbb{Z}^*}V_{4,s}s^2 \|D^+_s v^{\e}\|_\e^2\leq C\e^2\sum_{s\in\mathbb{Z}^*}V_{4,s}s^4\leq C\e^2.\label{eq..energy.consistency11}
    \end{eqnarray}
    We notice that $D^+_s v^{\e,+}_i-s\nabla v^+_i=D^+_s w_i+D^+_s v^+_i-s\nabla v^+_i$ and $|D^+_s v^+_i-s\nabla v^+_i-\frac{1}{2}\e s^2\nabla^2 v^+_i|\leq \frac{1}{6}\e^2 |s|^3 v_{3,s,i}$ (Recall Eq. \eqref{eq..u_is^k}). Using Lemma \ref{lem..u^k_is.L2norm}, we have $\|v_{3,s}\|_\e\leq C|s|^{1/2}$ and $\|\nabla^k v\|_\e\leq \|v_{k,1}\|_\e\leq C$, $k=1,2$.
    For the first term on the right hand side of Eq. \eqref{eq..energy.consistency1}, we have
    \begin{eqnarray}
        &&\left|\frac{\e}{2}\sum_{s\in\mathbb{Z}^*}V''(s)\sum_{i\in\mathbb{Z}}\left[(D^+_s v^{\e,+}_i)^2-(s\nabla v^+_i)^2\right]\right|\nonumber\\
        &\leq& \left|\frac{\e}{2}\sum_{s\in\mathbb{Z}^*}V''(s)\sum_{i\in\mathbb{Z}}(D^+_s w_i+D^+_s v^+_i-s\nabla v^+_i)(D^+_s v^{\e,+}_i+s\nabla v^+_i)\right|\nonumber\\
        &\leq& \frac{1}{2}\sum_{s\in\mathbb{Z}^*}V_{2,s}(\|D^+_s w\|_\e+\frac{1}{6}\e^2 |s|^3\|v_{3,s}\|_\e)(\|D^+_s v^\e\|_\e+|s|\|\nabla v\|_\e)\nonumber\\
        &&+\left|\frac{\e}{2}\sum_{s\in\mathbb{Z}^*}V''(s)\sum_{i\in\mathbb{Z}}\left(\frac{1}{2}\e s^2\nabla^2 v^+_i\right) D^+_s v^{\e,+}_i\right|+\left|\frac{\e}{2}\sum_{s\in\mathbb{Z}^*}V''(s)\sum_{i\in\mathbb{Z}}\left(\frac{1}{2}\e s^2\nabla^2 v^+_i\right) \nabla v^+_i\right|\nonumber\\
        &\leq& C\e^2\sum_{s\in\mathbb{Z}^*}V_{2,s}|s|^5+C \e^2 \sum_{s\in\mathbb{Z}^*}V_{2,s}s^4+0\leq C\e^2.\label{eq..energy.consistency12}
    \end{eqnarray}
    We have used the following fact that $\sum_{i\in\mathbb{Z}}\nabla^2 v^+_i \nabla v^+_i=\frac{1}{2}\sum_{i\in\mathbb{Z}}(\nabla^2 v^+_i \nabla v^+_i+\nabla^2 v^+_{-i} \nabla v^+_{-i})=0$,
    $\sum_{s\in\mathbb{Z}^*}V''(s)s^2D^+_s v^{\e,+}_i=\frac{1}{2}\sum_{s\in\mathbb{Z}^*}V''(s)s^2(D^+_s v^{\e,+}_i+D^+_{-s} v^{\e,+}_i)=\frac{\e}{2}\sum_{s\in\mathbb{Z}^*}V''(s)s^2(D^-_s D^+_s v^{\e,+}_i)$, and that
    \begin{eqnarray*}
        \left|\frac{\e}{2}\sum_{s\in\mathbb{Z}^*}V''(s)\sum_{i\in\mathbb{Z}}\left(\frac{1}{2}\e s^2\nabla^2 v^+_i\right) D^+_s v^{\e,+}_i\right|
        &\leq& \left|\frac{\e^3}{8}\sum_{s\in\mathbb{Z}^*}V''(s)s^2\sum_{i\in\mathbb{Z}}\nabla^2 v^+_i D^-_s D^+_s v^{\e,+}_i\right|\\
        &\leq& C \e^2 \sum_{s\in\mathbb{Z}^*}V_{2,s}s^4.
    \end{eqnarray*}

    Next, we estimate the misfit part.
    Thanks to Lemma \ref{lem..subspace.M}, we have $\|w^+\|_\e\leq \|w\|_{X_\e}\leq C\e^2$. Also recall that $\|v^+\|_\e\leq C$. Note that $v^{\e,+}_{i+s}+v^{\e,+}_i-2v^+_i=w^+_{i+s}+w^+_i+\e D^+_s v^+_i$ and $v^{\e,+}_{i+s}+v^{\e,+}_i-2v^+_{i+s}=w^+_{i+s}+w^+_i-\e D^+_s v^+_i$.
    Since $\sum_{s\in\mathbb{Z}}U'(s-\frac{1}{2})=0$ and the following series is absolutely summable, we have
    \begin{eqnarray*}
        \sum_{i\in\mathbb{Z}}\sum_{s\in\mathbb{Z}}
        U'(s-\frac{1}{2})(w^+_{i+s}+w^+_i)=2\sum_{i\in\mathbb{Z}}w^+_i\sum_{s\in\mathbb{Z}}U'(s-\frac{1}{2})=0.
    \end{eqnarray*}
    Now repeatedly applying Taylor theorem to $U$ leads to
    \begin{eqnarray}
        |2R_\text{mis}|&=&\left|\e\sum_{i\in\mathbb{Z}}\sum_{s\in\mathbb{Z}} \left[2U(s-\frac{1}{2}+v^{\e,+}_{i+s}+v^{\e,+}_i)-U(s-\frac{1}{2}+2v^+_i)-U(s-\frac{1}{2}+2v^+_{i+s})\right]\right|\nonumber\\
        &\leq&\left|\e\sum_{i\in\mathbb{Z}}\sum_{s\in\mathbb{Z}}
        \left[U'(s-\frac{1}{2}+2v^+_i)+U'(s-\frac{1}{2}+2v^+_{i+s})\right](w^+_{i+s}+w^+_i)\right|\nonumber\\
        &&+\left|\e\sum_{i\in\mathbb{Z}}\sum_{s\in\mathbb{Z}}
        \left[U'(s-\frac{1}{2}+2v^+_i)-U'(s-\frac{1}{2}+2v^+_{i+s})\right]\e D^+_s v^+_i
        \right|\nonumber\\
        &&
        +\left|\e\sum_{i\in\mathbb{Z}}\sum_{s\in\mathbb{Z}}
        \frac{1}{2}U_{2,s}\left[(w^+_{i+s}+w^+_i+\e D^+_s v^+_i)^2+(w^+_{i+s}+w^+_i-\e D^+_s v^+_i)^2\right]\right|
        \nonumber\\
        &\leq&\left|\e\sum_{i\in\mathbb{Z}}\sum_{s\in\mathbb{Z}}
        2U'(s-\frac{1}{2})(w^+_{i+s}+w^+_i)\right|
        +\left|\e\sum_{i\in\mathbb{Z}}\sum_{s\in\mathbb{Z}}
        U_{2,s}(2v^+_i+2v^+_{i+s})(w^+_{i+s}+w^+_i)\right|\nonumber\\
        &&+\e\sum_{i\in\mathbb{Z}}\sum_{s\in\mathbb{Z}}2U_{2,s}|\e D^+_{s}v^+_{i}|^2
        +C\e^2
        \nonumber\\
        &\leq&0+C\e^2+C\e^2+C\e^2\leq C \e^2.\label{eq..energy.consistency2}
    \end{eqnarray}
    Combining Eqs. \eqref{eq..energy.consistency1}, \eqref{eq..energy.consistency11}, \eqref{eq..energy.consistency12}, and \eqref{eq..energy.consistency2}, we obtain
    \begin{eqnarray*}
        |E_\text{a}[v^\e]-E^\text{app}_\text{PN}[v]|\leq C \e^2.
    \end{eqnarray*}

    3 Estimate $\left|E^\text{app}_\text{PN}[v]-E_\text{PN}[v]\right|$. Let $g(x)=\alpha (\nabla v^+(x))^2+\gamma(2v^+(x))$ for $x\in \mathbb{R}$. Then $g\in C^4$ and
    \begin{eqnarray*}
        g'(x)&=&2\alpha\nabla v^+\nabla^2 v^+ +2\gamma'(2v^+)\nabla v^+,\\
        g''(x)&=&2\alpha(\nabla^2 v^+)^2+2\alpha\nabla v^+\nabla^3 v^+ +4\gamma''(2v^+)(\nabla v^+)^2+2\gamma'(2v^+)\nabla^2 v^+.
    \end{eqnarray*}
    By Lemma \ref{lem..regularity.gamma}, we have $\|\gamma^{(k)}\|_{L^\infty}\leq C$, $k=1,2$. Thus
    \begin{eqnarray*}
        \max_{(i-1/2)\e\leq\xi\leq (i+1/2)\e}|g''(\xi)|
        \leq C\left\{(v_{2,1,i})^2+v_{1,1,i}v_{3,1,i}+(v_{1,1,i})^2+v_{2,1,i}\right\}.
    \end{eqnarray*}
    Finally, we apply Lemma \ref{lem..u^k_is.L2norm}:
    \begin{eqnarray*}
        \left|E^\text{app}_\text{PN}[v]-E_\text{PN}[v]\right|
        &\leq& \sum_{i\in\mathbb{Z}}\left|\int_{(i-1)\e}^{(i+1)\e}g(x)\D x-\e g(i\e)\right|\\
        &\leq& \frac{\e^3}{3}\sum_{i\in\mathbb{Z}}\max_{(i-1/2)\e\leq\xi\leq (i+1/2)\e}|g''(\xi)|\\
        &\leq&C\e^3\sum_{i\in\mathbb{Z}}\left\{(v_{2,1,i})^2+v_{1,1,i}v_{3,1,i}+(v_{1,1,i})^2+v_{2,1,i}\right\}\\
        &\leq&C\e^2.
    \end{eqnarray*}
\end{proof}

\section*{Appendix: Small parameter $\e$ calculated by atomistic and first principles calculations}

In this appendix, we calculate the small parameter $\e$ defined in Eq.~\eqref{eq..epsilon} in Sec.~\ref{sec..rescaling}
that characterizes the strength of the weak van der Waals interlayer interaction v.s. the strong covalent-bond intralayer interaction in the bilayer graphene, using the data of atomistic and first principles calculations \cite{Dai2016-p85410-85410,Zhou2015-p155438-155438}.

In the PN model for bilayer graphene in Ref.~\cite{Dai2016-p85410-85410},  the two dimensional $\gamma$-surface was fitted by a truncated trigonometric series as
\begin{eqnarray*}
    \gamma_{2\text{d}}(\phi,\psi)&=&c_0+c_1\left[\cos\frac{2\pi}{a}\left(\phi+\frac{\psi}{\sqrt{3}}\right)+\cos\frac{2\pi}{a}\left(\phi-\frac{\psi}{\sqrt{3}}\right)+\cos\frac{4\pi\psi}{\sqrt{3}a}\right]\\
    &&+c_2\left[\cos\frac{2\pi}{a}\left(\phi+\sqrt{3}\psi\right)+\cos\frac{2\pi}{a}\left(\phi-\sqrt{3}\psi\right)+\cos\frac{4\pi\phi}{a}\right]\\
    &&+c_3\left[\cos\frac{2\pi}{a}\left(2\phi+\frac{2\psi}{\sqrt{3}}\right)+\cos\frac{2\pi}{a}\left(2\phi-\frac{2\psi}{\sqrt{3}}\right)+\cos\frac{8\pi\psi}{\sqrt{3}a}\right]\\
    &&+c_4\left[\sin\frac{2\pi}{a}\left(\phi-\frac{\psi}{\sqrt{3}}\right)-\sin\frac{2\pi}{a}\left(\phi+\frac{\psi}{\sqrt{3}}\right)+\sin\frac{4\pi\psi}{\sqrt{3}a}\right]\\
    &&+c_5\left[\sin\frac{2\pi}{a}\left(2\phi-\frac{2\psi}{\sqrt{3}}\right)-\sin\frac{2\pi}{a}\left(2\phi+\frac{2\psi}{\sqrt{3}}\right)+\sin\frac{8\pi\psi}{\sqrt{3}a}\right],
\end{eqnarray*}
where $\{c_i\}_{i=1}^5$ are constants obtained by fitting the data of first principles calculations
\cite{Zhou2015-p155438-155438} as
\begin{eqnarray*}
    &&c_0=21.336\times10^{-3},\,\,c_1=-6.127\times10^{-3},\,\,c_2=-1.128\times10^{-3},
    \\&& c_3=0.143\times10^{-3},
    c_4=\sqrt{3}c_1,\,\,c_5=-\sqrt{3}c_3,
\end{eqnarray*}
where the units are $\mathrm{J}/\mathrm{m}^2$. On the other hand, the elasticity constants of each monolayer graphene, in the unit of $\mathrm{J}/\mathrm{m}^2$, are \cite{Dai2016-p85410-85410}
\begin{eqnarray*}
    C_{11}=312.67,\,\,C_{12}=91.66,\,\,C_{44}=110.40.
\end{eqnarray*}

In our one-dimensional case, $\gamma(\phi)=\gamma_{2\text{d}}(\phi,0)$ and $\alpha=C_{11}$.
Using the above values and Eq.~\eqref{eq..epsilon} in Sec.~\ref{sec..rescaling}, we have
\begin{eqnarray*}
    \e=\sqrt{\frac{a^2\frac{\partial^2\gamma_{2\text{d}}(0,0)}{\partial \phi^2}}{C_{11}}}\approx 0.0475.
\end{eqnarray*}
Thus it is reasonable to set $\e$ as a small parameter.

\section*{Acknowledgments}
This work was partially supported by the Hong Kong Research Grants Council General Research Fund 16313316.
The work of Ming was partially supported by the National Natural Science Foundation of China for Distinguished Young Scholars 11425106, and National Natural Science Foundation of China Grants 91630313, and by the support of CAS NCMIS.

\end{document}